\documentclass[11pt]{article}

\usepackage{amssymb,latexsym}
\usepackage{amsmath,amscd}
\usepackage{theorem}
\usepackage[all]{xy}
\usepackage{url}
\usepackage{stmaryrd}
\usepackage{bm} 
\usepackage{graphics}
\usepackage{color}
\usepackage{soul}

\title{\bf 
On the self-similarity of the norm one group
\\
of $\bm{p}$-adic division algebras
}

\author{
        Francesco Noseda  
        \\[0.1cm]
        Ilir Snopce 
}

\date{}

\newcommand{\bb}[1]{\mathbb{#1}}
\newcommand{\cl}[1]{\mathcal{#1}}

\newcommand{\mr}[1]{\mathrm{#1}}

\newcommand{\mfr}[1]{\mathfrak{#1}}

\newcommand{\rar}{\rightarrow}

\newcommand{\vep}{\varepsilon}

\newcommand{\les}{\leqslant}
\newcommand{\ges}{\geqslant}

\newcommand{\normal}{\trianglelefteqslant}

\newcommand{\ep}{\hfill $\square$} %end paragraph

\newtheorem{lemma} {Lemma} [section]

\newtheorem{theorem} [lemma] {Theorem}
\newtheorem{corollary} [lemma] {Corollary}
\newtheorem{definition}[lemma] {Definition}

\theorembodyfont{\rm} %%% This is to typeset examples
\newtheorem{remark}[lemma]{Remark}

\newenvironment{proof}{{\sc Proof:}}{%{\hspace*{\fill} $\square$\\}
\hfill $\square$}
\newenvironment{proof2}{{\sc Proof:}}{}

\numberwithin{equation}{section}

\theorembodyfont{\it}

\newtheorem{theoremx}{Theorem}

\newtheorem{conjx}[theoremx]{Conjecture}
\newtheorem{questionx}[theoremx]{Question}

\usepackage{array}
\usepackage[colorlinks]{hyperref}% rule: last package to be declared

\begin{document} 

\maketitle

\begin{abstract}
Let $p$ be a prime,
$D$ a finite dimensional noncommutative division $\bb{Q}_p$-algebra,
and $SL_1(D)$ the group of elements of $D$ of reduced norm $1$.
When the center of $D$ is $\bb{Q}_p$, 
we prove that no open subgroup of $SL_1(D)$ 
admits self-similar actions on regular rooted trees. 
Moreover, we prove results on $\bb{Z}_p$-Lie lattices that allow to deal
with the case
where the center of $D$ is bigger than $\bb{Q}_p$, 
and lead to the classification of the torsion-free $p$-adic analytic pro-$p$ groups $G$ 
of dimension less than $p$ with the property that all the nontrivial closed subgroups
of $G$ admit a self-similar action on a $p$-ary tree.
As a consequence, we obtain that 
a nontrivial torsion-free $p$-adic analytic pro-$p$ group $G$ 
of dimension less than $p$ is isomorphic
to the maximal pro-$p$ Galois group of a field that contains
a primitive $p$-th root of unity
if and only if all the nontrivial closed subgroups of $G$
admit a self-similar action on a regular rooted $p$-ary tree.

\end{abstract}

{
\let\thefootnote\relax
\let\thefootnote\relax
}

%%%%%%%%%%%%%%%%%%%%%%%
%%%%%%%%%%%%%%%%%%%%%%%

\section{Introduction}

Let $K$ be a $p$-adic field, that is, a finite field extension of $\bb{Q}_p$ for some prime $p$. 
For $D$ a finite dimensional central division algebra over $K$,
the group $SL_1(D)$ of elements of $D$ of reduced norm 1 
is algebraic, in the sense that it is the group of $K$-rational points of an algebraic group,
and it has the structure of compact $p$-adic analytic group.
The groups of type $SL_1(D)$ and their associated $\bb{Q}_p$-Lie algebras $sl_1(D)$
play an important role in the classification of, respectively, simple algebraic groups and 
simple Lie algebras defined over $p$-adic fields. 
More precisely,
if $G$ is the group of $K$-rational points of an anisotropic absolutely simple
simply connected algebraic group defined over $K$
then $G$ is isomorphic to the group $SL_1(D)$ for some $D$ \cite[Proposition 6]{BTgpalgsim}. 
Analogously, if $\cl{L}$ is an anisotropic absolutely simple $K$-Lie algebra then
$\cl{L}$ is isomorphic to $sl_1(D)$ for some $D$; 
see, for instance, \cite[Theorem (2), page 6]{Schoeneberg2014}. 
Due to their prominent role, these objects have been studied thoroughly;
see, for instance, \cite{RiehmNorm1, PRcohSL1, ErsFinPreSL1, Ers2CohN1_arxiv, Schoeneberg2014}; 
see also \cite{KloSL1}.
Moreover, we point out that the groups $SL_1(D)$ have strong connections
with the Morava stabilizer groups in stable homotopy theory;
see, for instance, \cite[Appendix 2]{RavComCob} and \cite{HennDuke98}.
The main goal of the present paper is to study  the groups $SL_1(D)$
from the point of view of their self-similar actions on regular rooted trees.

Groups that admit a faithful self-similar action on some regular 
rooted tree form an interesting class that contains many important examples 
such as the Grigorchuk 2-group \cite{Gri80} and the Gupta-Sidki $p$-groups \cite{GuSi83}.
More recently, there has been an intensive study on the self-similar actions of other 
families of groups, including 
iterated monodromy groups \cite{NekSSgrp},
finitely 
generated nilpotent groups \cite{BeSi07},  arithmetic groups \cite{Ka12}, 
finite $p$-groups \cite{Su11,BaFaFeVa20}, 
groups of type $\mr{FP}_n$ \cite{KoSi20},
and $p$-adic analytic pro-$p$ groups \cite{NS2019, NSGGD22, NS2022JGT}.

We say that a group $G$ is self-similar of index ${d}$,
where $d\ges 1$ is an integer,
if $G$ admits a faithful self-similar action on
a regular rooted $d$-ary tree in such a way that the action is transitive on the first level.
We say that $G$ is self-similar if it is self-similar of 
some index $d$; see, for instance, \cite{NekSSgrp} for a general treatment of self-similar actions. 
It is known that a group is self-similar of index $d$
if and only if it admits a simple virtual endomorphism of index 
$d$.
A {virtual endomorphism} of $G$ of index $d$
is a group homomorphism $\varphi:D\rar G$, where $D\les G$
is a subgroup of index $d$; the virtual endomorphism $\varphi$ is said to be simple
if there are no nontrivial normal subgroups $N$ of $G$ that are $\varphi$-invariant,
that is, such that $N\subseteq D$ and $\varphi(N)\subseteq N$.
\\

Throughout the paper we let $p$ be a prime.
The following theorem proves and significantly generalizes \cite[Conjecture E]{NS2019},
and it is proved in Section \ref{proofAB}.

\begin{theoremx}
\label{tSL1nss}
Let $D$ be a finite dimensional noncommutative central division $\bb{Q}_p$-algebra,
and $G$ an open subgroup of $SL_1(D)$.
Then $G$ is not self-similar. 
\end{theoremx}

\noindent
In particular, the theorem provides families of unsolvable torsion-free
$p$-adic analytic pro-$p$
groups of arbitrarily high dimension
that do not admit a self-similar action on a $p^k$-ary tree for any positive integer $k$. 
For $k=1$ we point out that, to the extent of our knowledge,
such families were only known for solvable groups 
(cf. \cite[Proposition 2.25]{NSGGD22} and \cite[Proposition A]{NS2019}). 
Moreover, we observe that the given proof of Theorem \ref{tSL1nss} 
does not work for the case of division algebras whose center is bigger than $\bb{Q}_p$ 
(see Remark \ref{rnoncen}); for this case we will apply Lie techniques, as explained below.
In doing so, we obtain results that are interesting on their own 
(Theorems \ref{tMMLLZ_p.2}, \ref{taim2.2}, and \ref{tmainalt})
and that may be applied to prove Theorem \ref{tshssL} and Theorem \ref{conjE}.

\bigskip

In \cite{NS2019} we initiated the study of self-similar actions of $p$-adic analytic pro-$p$ groups
by exploiting the Lazard's correspondence, 
an isomorphism of categories between the category
of saturable  finitely generated pro-$p$ groups and the category of saturable $\bb{Z}_p$-Lie lattices;
see \cite[IV (3.2.6)]{Laz65} and, for instance, \cite{GSpsat} 
and \cite{GSKpsdimJGT} for more details.
We recall that the Lazard's correspondence 
restricts to an isomorphism between the category
of uniform pro-$p$ groups and the category of powerful $\bb{Z}_p$-Lie lattices;
see \cite{DixAnaProP}.
A $\bb{Z}_p$-Lie lattice is a finitely generated free $\bb{Z}_p$-module 
endowed with a structure of Lie algebra over $\bb{Z}_p$.
Let $L$ be a $\bb{Z}_p$-Lie lattice and $k\in\bb{N}$.
A virtual endomorphism
of $L$ of index $p^k$ is an algebra  homomorphism $\varphi:M\rar L$
where $M$ is a subalgebra of $L$ of index $p^k$,
where the index is taken with respect to the additive group structure.
An ideal $I$ of $L$ is said to be $\varphi$-invariant if
$I\subseteq M$ and $\varphi(I)\subseteq I$.
We say that a virtual endomorphism $\varphi$ is simple if there are no nonzero ideals $I$
of $L$ that are $\varphi$-invariant.
Finally, 
we say that $L$ is {self-similar of index} ${p^k}$ if there exists
a simple virtual endomorphism of $L$ of index $p^k$,
and we say that $L$ is self-similar if 
it is self-similar of index $p^k$ for some $k$.
\\

The following is a version of Theorem \ref{tSL1nss} for $\bb{Z}_p$-Lie lattices,
and it is proved in Section \ref{proofAB}.

\begin{theoremx}
\label{tsl1nss}
Let $D$ be a finite dimensional noncommutative central division $\bb{Q}_p$-algebra,
$\Delta$ the ring of integers of $D$, and $L$ a $\bb{Z}_p$-Lie subalgebra
 of $sl_1(\Delta)$ of finite index.
Then $L$ is not self-similar.
\end{theoremx}

\bigskip

We now turn to the case where the center of $D$ may be bigger than $\bb{Q}_p$.
More precisely,
let $D$ be a finite dimensional central division $K$-algebra,
where $K$ is a finite field extension of $\bb{Q}_p$.
Let $d$ and $e$ be the degree and ramification index of $K/\bb{Q}_p$, respectively, 
and let $n$ be the index of $D$ over $K$, that is $\mr{dim}_K(D)=n^2$.
Let $\cl{O}_K$ be the integral closure of $\bb{Z}_p$ in $K$, and
$\Delta$ be the unique maximal $\cl{O}_K$-order in $D$.
Let $SL_1(D)$ be the set of elements of 
$D$ (equivalently, of $\Delta$) of reduced norm equal to $1$,
and observe that $SL_1(D)$ is a compact $p$-adic analytic group 
of dimension $d(n^2-1)$.
Let $\mfr{p}$ be the unique maximal two-sided ideal of $\Delta$, and
for $l\in\bb{N}$ define $SL_1^l(D)=SL_1(D)\cap (1+\mfr{p}^l)$,
which is an open subgroup of $SL_1(D)$.
The following theorem is proved in Section \ref{proofC}.

\begin{theoremx}
\label{tmainalt}
Assume that $p\ges  d(n^2-1)$,
and let $m\ges 1$ be an integer. 
Then 
$SL_1^{mne}(D)$
is a uniform pro-$p$ group, and the following holds.
\begin{enumerate}
\item 
\label{tmainalt1}
$SL_1^{mne}(D)$ is not self-similar of index $p$.
\item 
\label{tmainalt2}
Assume that $n\neq 2$, and 
let $H$ be a maximal proper subgroup of $SL_1^{mne}(D)$.
Then $[H,H]=[SL_1^{mne}(D),SL_1^{mne}(D)]$.
\end{enumerate}
\end{theoremx}

%%%%%%%%%%%%%%%%%%%
\bigskip 

Theorem \ref{tSL1nss} and Theorem \ref{tmainalt}
provide evidence for the following conjecture.

\begin{conjx}
\label{conjsl1_v2}
Let $D$ be a finite dimensional noncommutative division $\bb{Q}_p$-algebra,
and $G$ an open subgroup of $SL_1(D)$. 
Then $G$ is not self-similar.
\end{conjx}

\bigskip 

A finitely generated pro-$p$ group $G$ is said to be 
{strongly hereditarily self-similar of index ${p}$}
if $G$ and all of its nontrivial 
finitely generated closed subgroups are self-similar of index 
$p$.
The groups that are strongly hereditarily self-similar of index $p$
were classified inside the class of solvable torsion-free $p$-adic analytic pro-$p$ groups
of dimension strictly less than $p$ \cite[Theorem A]{NSGGD22}.
In an analogous fashion to what is done for groups, we say that
a $\bb{Z}_p$-Lie lattice $L$ is {strongly hereditarily self-similar of index} ${p}$
if $L$ and all of its nonzero subalgebras are self-similar of index $p$.
For $p\ges 3$, the Lie lattices that are  {strongly hereditarily self-similar of index} ${p}$
were classified inside the class of solvable $\bb{Z}_p$-Lie lattices, without restriction
on the dimension \cite[Theorem 2.34]{NSGGD22}.
In this paper we show that the aforementioned classifications work over much wider classes, 
by not restricting to solvable groups and Lie lattices.

\begin{theoremx}
\label{tshssL}
Assume $p\ges 5$,
let $L$ be a $\bb{Z}_p$-Lie lattice, and define $d=\mr{dim}(L)$.
Then $L$ is strongly hereditarily self-similar of index $p$ if  and only if
$L$ is isomorphic to one of the following Lie lattices:
\begin{enumerate}
\item 
the abelian Lie lattice $\bb{Z}_p^d$, with $d\ges 1$;
\item 
the metabelian Lie lattice $\bb{Z}_p\ltimes \bb{Z}_p^{d-1}$,
where the canonical generator of $\bb{Z}_p$ acts on $\bb{Z}_p^{d-1}$
by multiplication by the scalar $p^s$, with $d\ges 2$ and $s\in\bb{N}$. 
\end{enumerate}
\end{theoremx}

\bigskip

The following theorem proves \cite[Conjecture E]{NSGGD22}.

\begin{theoremx}
\label{conjE}
Let $p$ be a prime and 
$G$ a torsion-free $p$-adic analytic pro-$p$ group. 
Define $d=\mr{dim}(G)$, and assume that $p>d$. 
Then $G$ is strongly hereditarily self-similar of index $p$ if and only if 
$G$ is isomorphic to one of the following groups:
  \begin{enumerate}
  \item 
  the abelian pro-$p$ group $\mathbb{Z}_p^d$, with $d\ges 1$;
  \item  
  the metabelian pro-$p$ group
  $\bb{Z}_p\ltimes \bb{Z}_p^{d-1}$, 
  where
  the canonical generator of 
  $\bb{Z}_p$ acts on $\bb{Z}_p^{d-1}$
  by  multiplication by the scalar $1+p^s$,
  with $d\ges 2$ and integer $s\ges 1$. 
  \end{enumerate} 
\end{theoremx}

\noindent 
The proofs of Theorem \ref{tshssL} and Theorem \ref{conjE} are given in Section \ref{proofGHI}.
The proof of Theorem \ref{conjE} relies on Theorem \ref{tshssL} and \cite[Proposition 3.1]{NSGGD22},
the latter being the key to exploit Lie methods.
We point out that, for the proof of Theorem \ref{tshssL},
one applies classical results
on Lie algebras over fields of characteristic zero to reduce the problem
to the study of the $\bb{Z}_p$-Lie lattices associated with the finite dimensional 
division $\bb{Q}_p$-algebras introduced above;
then, one can apply Theorem \ref{taim2.2}.

\bigskip

The groups that appear in Theorem \ref{conjE} have been object of study
in the past few decades,
since they are connected with several topics, such as,
maximal pro-$p$ Galois groups of fields \cite{Ware92}, 
pro-$p$ groups that have constant generating number 
on open subgroups \cite{KSjalg11},
Bloch-Kato pro-$p$ groups \cite{Quad14},
hereditarily uniform pro-$p$ groups
\cite{KSquart14},
and Frattini-injective pro-$p$ groups
\cite{ST2020frattini}. 
With the aid of Theorem \ref{conjE} we can prove 
a rather strong generalization of \cite[Theorem 3.9]{NSGGD22}.
The proof of Theorem \ref{thBext} is given in Section \ref{proofGHI}.

\begin{theoremx}
\label{thBext}
Let $G$ be a nontrivial 
torsion-free $p$-adic analytic pro-$p$ group, and assume that $p > \textrm{dim}(G)$. 
Then the following are equivalent. 
\begin{enumerate}
\item 
\label{peqc.1}
$G$ is strongly hereditarily self-similar of index $p$.

\item 
\label{peqc.2}
$G$ is isomorphic to the maximal pro-$p$ Galois group of some 
field that contains a primitive $p$-th root of unity.
\item  
\label{peqc.3}
$G$ has constant generating number on open subgroups.

\item 
\label{peqc.4}
$G$ is a Bloch-Kato pro-$p$ group.

\item  
\label{peqc.5}
$G$ is a hereditarily uniform pro-$p$ group.
\item 
\label{peqc.6}
$G$ is a Frattini-injective pro-$p$ group.
\end{enumerate} 
\end{theoremx}

\bigskip

Finally,
we say that a finitely generated profinite group is 
hereditarily self-similar of index $d$ if all
of its open subgroups are self-similar of index $d$. 
In \cite[Section 4]{NSGGD22} a few problems on the subject were raised;
we close the introduction by adding another relevant question.

\begin{questionx}
Is there an unsolvable compact $p$-adic analytic group which
is hereditarily self-similar of some index?
\end{questionx}

\vspace{5mm}

\noindent
\textbf{Notation.} 
The set $\bb{N}=\{0,1,2,...\}$ of natural numbers is assumed to contain $0$.
Throughout the paper we let $p$ be a prime,
and we denote by $\bb{Z}_p$ the ring of $p$-adic integers, by $\bb{Q}_p$ the field
of $p$-adic numbers, by $\bb{F}_p$ the residue field of $\bb{Z}_p$,
and by $v_p:\bb{Q}_p\rar \bb{Z}\cup\{\infty\}$ the $p$-adic valuation. 
The expression $x\equiv_a y$ means that $x$ is equivalent to $y$ modulo $a$,
where the context where the equivalence is applied will always be clear.
The characteristic of a field $K$ is denoted by $\mr{ch}(K)$.
When $R$ is a commutative ring, 
we denote by $R^*$ the associated multiplicative group;
moreover, when $L$ is an $R$-Lie algebra
and $M$ is an $R$-submodule of $L$, $[M,M]$ denotes the
set of finite sums of brackets of elements of $M$,
and we observe that $[M,M]$ coincides with the 
$R$-submodule of $L$ generated by such brackets.

\vspace{5mm}

\noindent
\textbf{Acknowledgments.}
The authors thank Gopal Prasad for feedback on a question on division algebras. 
The first author also thanks José J. Ramón Marí for a technical suggestion.

%%%%%%%%%%%%%%%%%%%

\section{Proof of Theorem \ref{tSL1nss} and Theorem \ref{tsl1nss}}
\label{proofAB}

\textbf{Proof of Theorem \ref{tSL1nss}}.
Let $H$ be an open subgroup of $G$, 
and let $\varphi : H\rar G$ be a group homomorphism.
We have to show that there exists a nontrivial normal subgroup $N$
of $G$ such that $N\subseteq H$ and $\varphi(N)\subseteq N$.
We first prove the claim that the restriction $\varphi_U$ of $\varphi$ to an appropriate 
open subgroup $U$ of $H$ is either injective or it has an open kernel. 
Indeed, let $U$ be an open uniform pro-$p$ subgroup of $H$,
and assume that $\mr{ker}(\varphi_U)$ is not open. 
Observe that the $\bb{Q}_p$-Lie algebra associated with $U$
is canonically isomorphic to $sl_1(D)$,
which is a simple $\bb{Q}_p$-Lie algebra; see Remark \ref{rsimple}.
Moreover, the dimension $r$ of $\mr{ker}(\varphi_U)$ is strictly less than $\mr{dim}(U)$.
If $r=0$ then $\mr{ker}(\varphi_U)$ is trivial, since $U$ is torsion free.
If $r>0$ then we get a contradiction as follows. 
Let $V$ be a characteristic open uniform pro-$p$ subgroup of $\mr{ker}(\varphi_U)$.
Then the $\bb{Q}_p$-Lie algebra associated with $V$ gives a nontrivial proper ideal of
$sl_1(D)$, a contradiction, and the claim is proved. 
Now, we proceed by considering the two cases.
First, assume that $\mr{ker}(\varphi_U)$ is open.
Since $G$ is profinite,
there exists $N\normal_o G$ such that $N\subseteq \mr{ker}(\varphi_U)$.
Since $D$ is noncommutative, the groups involved have strictly positive dimension,
so that  $N$ is nontrivial; moreover, 
$N \subseteq H$ and $\varphi(N)=\{1\}\subseteq N$.
Now, assume that $\varphi_U$ is injective. 
Then $\varphi_U$ is a virtual automorphism of $G$,
and we can consider its class in the commensurator of $G$.
By \cite[Lemma XI.14]{KGPlinearppg} and \cite[Theorem 3.12]{BEW2011},
there exist $U'\les_o U$ and $g\in D^*$ 
such that, for all $x\in U'$, we have $\varphi(x)=gxg^{-1}$.
There exists $m\in\bb{N}$ such that $N:= SL_1^m(D)\subseteq U'$,
and we observe that $N$ is normal in $D^*$.
Hence, $N$ is a nontrivial normal subgroup of $G$ and $N\subseteq H$.
Moreover, for all $x\in N$, we have $\varphi(x)= gxg^{-1}\in
gNg^{-1}= N$, so that $\varphi(N)\subseteq N$. \ep
\\

\noindent 
\textbf{Proof of Theorem \ref{tsl1nss}}.
Let $\varphi :M\rar L$ be a virtual endomorphism of $L$. 
We have to show that $\varphi$ is not simple.
Observe that we have identifications
$M\otimes_{\bb{Z}_p}\bb{Q}_p=sl_1(D)$ 
and 
$L\otimes_{\bb{Z}_p}\bb{Q}_p=sl_1(D)$;
in particular, $L$ is just infinite 
and we can assume that $\varphi$ is injective;
see Remark \ref{rsimple}, Lemma \ref{lsjipid2}, and Lemma \ref{lsiminj2}.
We will show that there exists a nontrivial ideal $I$ of $L$ such that $I\subseteq M$ and 
$\varphi(I)\subseteq I$.
Observe that 
one can view $\varphi$ as the restriction to $M$ of an automorphism 
$\widehat{\varphi}$ of the $\bb{Q}_p$-Lie algebra $sl_1(D)$.
By \cite[Lemma XI.14]{KGPlinearppg}, there exists $a\in D^*$ such that,
for all $x\in sl_1(D)$, one has $\widehat{\varphi}(x)=axa^{-1}$. 
Since $M$ has finite index in $sl_1(\Delta)$, there exists $m\in\bb{N}$
such that $I:=p^msl_1(\Delta)\subseteq M$.
Hence, $I$ is a nontrivial ideal of $L$ that is included in $M$.
Let $x\in I$. We have to show that $\varphi(x)\in I$.
There exists $y\in sl_1(\Delta)$ such that $x = p^m y$.
We have $\varphi(x)= axa^{-1}= p^m aya^{-1}$.
By the properties of the reduced norm and trace, 
and by recalling that $\Delta=\{b\in D: N(b)\in\bb{Z}_p\}$,
where $N:D\rar \bb{Q}_p$ is the reduced norm,
one can see that $aya^{-1}$ is an element of $sl_1(\Delta)$,
so that $\varphi(x)\in I$, as desired.
\ep

\begin{remark}
\label{rnoncen}
We argue that the proofs of Theorem \ref{tSL1nss} and Theorem \ref{tsl1nss}
do not work when the center $K$ of $D$ is not $\bb{Q}_p$.
Consider the diagram 
$$
PGL_1(D) 
\rar
\mr{Aut}_{K}(sl_1(D))
\rar
\mr{Aut}_{\bb{Q}_p}(sl_1(D))
\rar
\mr{Comm}(SL_1(D)),
$$ 
where $\mr{Comm}$ denotes the commensurator,
the left arrow is the isomorphism given by \cite[Lemma XI.14]{KGPlinearppg},
the middle arrow is the inclusion, and 
the right arrow 
is the isomorphism given by \cite[Theorem 3.12]{BEW2011}.
The inclusion in the middle is in general strict
when $K\neq \bb{Q}_p$; see, for instance, \cite[Lemma XI.13]{KGPlinearppg}.
Hence, not all elements of 
$\mr{Aut}_{\bb{Q}_p}(sl_1(D))$
and 
$\mr{Comm}(SL_1(D))$
are given by conjugations by elements of $D^*$,
as it was used in the proofs.
\end{remark}

%%%%%%%%%%%%%%%%%%%%%%%%%

\section{Proof of Theorem \ref{tmainalt}}
\label{proofC}

The proof of Theorem \ref{tmainalt} relies on 
Theorem \ref{tMMLLZ_p.2} and Theorem \ref{taim2.2}, which are stated below
and have independent interest.
The proof of these theorems is rather long,
so that the present section is divided into several subsections.
Sections \ref{sliecyc}, \ref{sdvr}, and \ref{sedvr} 
contain general results and technical lemmas that
do not depend on the division algebra $D$ under consideration.
The structure of $D$ is recalled in Section \ref{divalg}.
The two main results of these preliminary sections, Theorem \ref{tn2} and Theorem \ref{tMMLLO_K},
are proved in Section \ref{sn2} and Section \ref{sO_K}, respectively.
Finally, Theorems  \ref{tMMLLZ_p.2}, \ref{taim2.2}, and \ref{tmainalt}
are proved in Section \ref{proofCDEss}.
\\

In the context of Theorem \ref{tmainalt} stated in the Introduction,
let $f$ be the inertia degree of $K/\bb{Q}_p$,
and let $sl_1(\Delta)$ be the set of elements of $\Delta$ of reduced trace equal to $0$.
Observe that $sl_1(\Delta)$ is a $\bb{Z}_p$-Lie lattice of dimension $d(n^2-1)$.

\begin{theorem}
\label{tMMLLZ_p.2}
Assume that $f\ges 2$ holds or both $n\ges 3$ and $(p,n)\neq (3,3)$ hold.
Let $N$ be a maximal proper $\bb{Z}_p$-submodule of $sl_1(\Delta)$.
Then $[N,N]=[sl_1(\Delta),sl_1(\Delta)]$.
\end{theorem}

\begin{theorem}
\label{taim2.2}
Assume that $f\ges 2$ or $(p,n)\not\in \{(2,2),(3,3)\}$, and
let $m\in\bb{N}$. 
Then $p^msl_1(\Delta)$ is not self-similar of index $p$.
\end{theorem}

%%%%%%%%%%%%%%%%
\subsection{Lemmas in cyclic extensions}
\label{sliecyc}

For an arbitrary field $K$,
let $F/K$ be a finite cyclic Galois field extension, let $n$ be the degree of $F/K$,
and let $\vartheta$ be a generator of $\mr{Gal}(F/K)$.
Let $T:F\rar K$ be the trace map, and let $F^0=\mr{ker}(T)$.
Since $F/K$ is finite and separable,
the trace map is surjective, and the trace form,
$(\alpha,\beta)\in F\times F\mapsto T(\alpha\beta)$,
is nondegenerate.
In particular, $F^0$ is
a $K$-vector subspace of $F$ of dimension $n-1$.
We will prove some results
that will be applied, in Section \ref{sO_K},
to the residue field extension $\kappa_F/\kappa_K$ that will be introduced 
in Section \ref{divalg}.

\begin{lemma}
\label{ltaneqa}
Let $0<j<n$, and assume that $(\mr{ch}(K),n)\neq (2,2)$. 
Then there exists $\alpha\in F^0$ such that $\vartheta^j(\alpha)\neq \alpha$.
\end{lemma}

\begin{proof}
Observe that $\vartheta^j\neq \mr{id}$.
For $\mr{ch}(K)\nmid n$ the lemma is clear, since $F=K\oplus F^0$
and $\vartheta^j$ fixes the elements of $K$, so that $\vartheta^j$ does not
fix all the elements of $F^0$.
Below we will prove the lemma for $n\ges 3$. Since, under the assumptions,
$\mr{ch}(K)\,|\, n$ implies $n\ges 3$, the lemma is proven in all the cases.

Assume that $n\ges 3$. Let $S=\vartheta^j-\mr{id}$, a $K$-linear map $F\rar F$.
Let $F_1=\mr{ker}(S)$, and observe that $F_1$ is the fixed field of $\vartheta^j$.
Hence, $F_1$ is an intermediate field of the extension $F/K$, and $F_1\neq F$. 
Assume by contradiction that $\vartheta^j$ fixes the elements of $F^0$, in other words,
$F^0\subseteq F_1$. Then $F^0=F_1$, which implies that $n-1$ divides $n$, a contradiction
since $n\ges 3$.
\end{proof}

\begin{lemma}
\label{lcyc0}
Let 
$\alpha\in F$ and $\alpha\neq 0$. 
Then  
$\{\alpha\vartheta(\beta)-\beta\vartheta^{-1}(\alpha):\, \beta \in F\}= F^0$.
\end{lemma}

\begin{proof}
Let  
$V=\{\alpha\vartheta(\beta)-\beta\vartheta^{-1}(\alpha):\, \beta \in F\}$.
Since $\vartheta$ preserves the trace, one can see that $V\subseteq F^0$.
Let $h:F\rar F$ be the $K$-linear map defined by 
$h(\beta)= \alpha\vartheta(\beta)-\beta\vartheta^{-1}(\alpha)$,
and observe that $V$ is the range of $h$.
Since the range of $h$ has dimension less than $n$, 
there exists $\beta_0\in \mr{ker}(h)$ such that $\beta_0\neq 0$.
An easy computation shows that, for all $\beta\in F$,
$\beta\in\mr{ker}(h)$ if and only if $\vartheta(\beta/\beta_0)=\beta/\beta_0$.
Since the fixed field of $\vartheta$ is $K$, we have $\mr{ker}(h)=\beta_0K$.
Hence, $V$ has dimension $n-1$ over $K$, and the lemma follows.
\end{proof}

\begin{remark}
\label{rcyc0}
By applying Lemma \ref{lcyc0} with the generator $\vartheta^{-1}$ in place of
$\vartheta$, one sees that if $\beta\in F$ and $\beta\neq 0$ then 
$\{\alpha\vartheta(\beta)-\beta\vartheta^{-1}(\alpha):\, \alpha \in F\}= F^0$.
\end{remark}

\bigskip

The following is the additive form Hilbert's Theorem 90.

\begin{lemma}
\label{lcyc1}
We have  
$\{\vartheta(\alpha)-\alpha:\, \alpha \in F\}= F^0$.
\end{lemma}

\begin{proof}
The $K$-linear map $\vartheta-\mr{id}:F\rar F$ has kernel of
dimension $1$ and range included in $F^0$. 
\end{proof}

\begin{lemma}
\label{lcyc2}
Let $k$ be an integer such that $k+1\not\equiv_n 0$.
Then 
$\mr{span}_K(\alpha\vartheta(\beta)-\beta\vartheta^k(\alpha):\, \alpha,\beta \in F)= F$.
\end{lemma}

\begin{proof}
Define $V=\mr{span}_K(\alpha\vartheta(\beta)-\beta\vartheta^k(\alpha):\, \alpha,\beta \in F)$,
and observe that $\{\vartheta(\beta)-\beta:\beta\in F\}= F^0$, by Lemma \ref{lcyc1}. 
By taking $\alpha=1$ one sees that $F^0\subseteq V$. 
It is enough to show that there exists $\gamma\in V$ such that $T(\gamma)\neq 0$.
Assume by contradiction that $T(\alpha\vartheta(\beta)-\beta\vartheta^k(\alpha))=0$
for all $\alpha,\beta \in F$. By substituting $\alpha$ with $\vartheta(\alpha)$, we get 
$T(\vartheta(\alpha\beta)-\beta\vartheta^{k+1}(\alpha))=0$ for all $\alpha,\beta$.
Hence, for all $\alpha,\beta$, we have 
$T(\alpha\beta)=T(\vartheta(\alpha\beta))=T(\beta\vartheta^{k+1}(\alpha))$,
so that $T([\vartheta^{k+1}(\alpha)-\alpha]\beta)=0$.
Since the trace form is nondegenerate, we get $\vartheta^{k+1}(\alpha)=\alpha$
for all $\alpha\in F$. This is a contradiction since $\vartheta^{k+1}\neq \mr{id}$.
\end{proof}

\bigskip

For the applications of Section \ref{sO_K}
we need a basis of $F$ over $K$
with the properties given in the following lemma.
The ensuing remarks will also be applied in that section.

\begin{lemma}
\label{lbascyc}
There exists a basis $(\tau_i)_{0\les i < n}$ of $F$ over $K$
such that the following holds.
\begin{enumerate}
\item 
$(\tau_i)_{1\les i < n}$ is a basis of $F^0$ over $K$.

\item 
If $\mr{ch}(K)\nmid n$ then $\tau_0=1$.

\item 
If $\mr{ch}(K) \,|\, n$ then  $\tau_1=1$.

\item 
\label{lbascyc4}
If $n\ges 2$ and $(\mr{ch}(K),n)\neq (2,2)$ then 
$\tau_{n-1}\not\in K$.
\end{enumerate}
\end{lemma}

\begin{proof}
There exists a basis $(\gamma_i)_{0\les i <n}$ of $F$ over $K$ such that $\gamma_0=1$.
There exists $i_0$ such that $T(\gamma_{i_0})\neq 0$.
In case $\mr{ch}(K)\nmid n$, take $i_0=0$.
We define 
$$
\begin{array}{ll}
\tau_{0} = \gamma_{i_0}
& 
\\[3pt]
\tau_{i}  = \gamma_{i-1}- T(\gamma_{i-1})T(\gamma_{i_0})^{-1}\gamma_{i_0}
&
\mbox{if } 0< i \les i_0
\\[3pt]
\tau_{i}  = \gamma_{i}- T(\gamma_{i})T(\gamma_{i_0})^{-1}\gamma_{i_0}
&
\mbox{if } i > i_0,
\end{array}
$$
and one can see that $(\tau_i)_{0\les i <n}$ satisfies the desired properties.
\end{proof}

\begin{remark}
\label{rxisigma}
Let $(\tau_i)_{0\les i < n}$ be a basis of $F$ over $K$
with the properties stated in Lemma \ref{lbascyc}.
Define $\xi_i=\tau_i-\vartheta(\tau_i)$ for $0\les i < n$, 
and recall Lemma \ref{lcyc1}.
Define $i_0=0$ when 
$\mr{ch}(K)\nmid n$,
and  
$i_0=1$ when $\mr{ch}(K)\,|\, n$.
Observe that $\xi_{i_0}=0$.
Then the set 
$(\xi_i)_{0\les i < n,\, i\neq i_0}$
is linearly independent over $K$.
\end{remark}

\begin{remark}
\label{rcyc2}
Let $V$ be a $K$-vector subspace of $F$
of dimension $n-1$.
Observe that if $\alpha\in F$ and $\alpha V=V$ then
$\alpha\in K$. In fact, $V$ is a vector space over 
the subfield 
$K(\alpha)$ generated by $\alpha$ over $K$, and we have
$\mr{dim}_K(V)=\mr{dim}_{K(\alpha)}(V)\cdot [K(\alpha):K]$,
from which we get $[K(\alpha):K]=1$.
Now, 
let $(\xi_0, \xi_1)$ be an ordered set of elements of $F$ that is 
linearly independent over $K$. 
Since $\xi_0/\xi_1\not\in K$, by the above observation we have 
$\xi_0 V\neq \xi_1 V$. 
Hence, $\xi_0 V+\xi_1 V=F$.

\end{remark}

%%%%%%%%%%%%%%%%%%%%
\subsection{Lemmas over discrete valuation rings}
\label{sdvr}

This section is divided into three parts. 
In the first, we discuss the notion of just-infinite Lie-lattice;
in the second, we introduce the convenient terminology of 
relative invariant exponents and $s$-invariants;
and in the third, we introduce a notation for the parametrization of 
maximal proper submodules of a lattice.

Let $R$ be a commutative ring. An $R$-lattice is a finitely generated free $R$-module.
When an $R$-lattice is endowed with the structure of an $R$-Lie algebra, we call it an $R$-Lie lattice.
For our purposes, even when proving or recalling general properties, 
it is natural to work over a principal ideal domain, so that a submodule of a lattice is a lattice.
So, we assume that $R$ is a principal ideal domain, and we denote by $F$ the field of fractions of $R$.
When $V$ is a finite dimensional $F$-vector space,
a full $R$-lattice in $V$ is a finitely generated $R$-submodule of $V$
whose span over $F$ is $V$.
Observe that, 
since $R$ is a principal ideal domain, a full $R$-lattice is an $R$-lattice.
Finally, if $L$ is an $R$-lattice then one can identify
$L$ with a full $R$-lattice in $L\otimes_R F$.

\bigskip

For the first part of the section, we say that an $R$-Lie lattice $L$ is just infinite if 
it is nonzero and, for all nonzero ideals $I$ of $L$,
we have $\mr{dim}_R(I)=\mr{dim}_R(L)$.
Also, $L$ is said to be hereditarily just infinite
if all subalgebras $M$ of $L$ with $\mr{dim}_R(M)=\mr{dim}_R(L)$
are just infinite.

\begin{remark}
\label{rsubz}
Let $L$ be a $ \bb{Z}_p$-lattice and $M$ be a $\bb{Z}_p$-submodule of $L$.
Then $\mr{dim}_{\bb{Z}_p}(M)=\mr{dim}_{\bb{Z}_p}(L)$ if and only if $[L:M]$ is finite,
where the index is taken with respect to the additive group structure. 
Also, $M$ is a maximal proper $\bb{Z}_p$-submodule of $L$ if and only if $[L:M]=p$.
\end{remark}

\begin{lemma}
\label{lsjipid2}
Let $L$ be an $R$-Lie lattice, and let 
$\cl{L}:=L\otimes_{R}F$ be the associated
$F$-Lie algebra. 
Then the following are equivalent.
\begin{enumerate}
\item
\label{lsjipid2.1}
$L$ is just infinite.

\item
\label{lsjipid2.2}
$\cl{L}$ is nonzero and it has no nonzero proper ideals.

\item
\label{lsjipid2.3}
$L$ is hereditarily just infinite.
\end{enumerate}
\end{lemma}

\begin{proof}
By identifying $L\subseteq \cl{L}$,
one can prove that $\cl{I}=\cl{L}$ for all nonzero ideals $\cl{I}$ of $\cl{L}$
if and only if
$\mr{dim}_R(I)=\mr{dim}_R(L)$ for all nonzero ideals $I$ of $L$.
This proves the equivalence between (\ref{lsjipid2.1}) and (\ref{lsjipid2.2}).
The proof may be completed after observing that
if $M$ is an $R$-submodule of $L$ such that $\mr{dim}_R(M)=\mr{dim}_R(L)$
then the $F$-Lie algebras $M\otimes_R F$ and $\cl{L}$ may
be identified.
\end{proof}

\bigskip

The following lemma is a version of \cite[Lemma 10]{NS2022JGT}.

\begin{lemma}
\label{lsiminj2}
Let $L$ be a just-infinite $\bb{Z}_p$-Lie lattice, 
and let $\varphi$ be a virtual endomorphism of $L$
that is not injective.
Then $\varphi$ is not simple. 
\end{lemma}

\begin{lemma}
\label{lMMLLnss2}
Let $L$ be a just-infinite $\bb{Z}_p$-Lie lattice with $\mr{dim}_{\bb{Z}_p}(L)\ges 2$, and let $k\in\bb{N}$.
Assume that for all $\bb{Z}_p$-submodules $M$ of $L$ of index $p^k$ we have $[M,M]=[L,L]$.
Then $p^mL$ is not self-similar of index $p^k$ for all $m\in\bb{N}$.
\end{lemma}

\begin{proof}
Let $m\in\bb{N}$ and $\varphi:M\rar p^m L$ be a virtual endomorphism of $p^mL$ of index $p^k$. 
We have to show that $\varphi$ is not simple.
Observe that $p^mL$ is just infinite (Lemma \ref{lsjipid2})
and $\mr{dim}_{\bb{Z}_p}(p^mL)\ges 2$;
as a consequence, $p^mL$ is not abelian.
By Lemma \ref{lsiminj2},
we can assume that $\varphi$ is injective.
Let $M'=\varphi(M)$, so that $\varphi$ induces an isomorphism
of $\bb{Z}_p$-Lie algebras $M\simeq M'$.
By \cite[Corollary 8]{NSindexstable2021},
$M'$ has index $p^k$ in $p^mL$.
Hence, $p^{-m}M$ and $p^{-m}M'$ are submodules of $L$ of index $p^k$.
By assumption, 
$[p^{-m}M, p^{-m}M]=[L,L]=[p^{-m}M',p^{-m}M']$.
Hence, 
$[M, M]=[p^mL,p^mL]=[M',M']$.
It follows that $[p^mL,p^mL]$ is a nonzero ideal of $p^mL$ that
is $\varphi$-invariant, and we conclude that $\varphi$ is not simple.
\end{proof}

\bigskip

For the second part of the section,
we assume that $R$ is a discrete valuation ring, and we let $\varrho\in R$ be a uniformizing parameter. 
For the basic properties of the Hermite and Smith canonical forms,
the reader may consult, for instance, \cite[Section 5.2]{AWalgmod}.
Recall that $F$ denotes the field of fractions of $R$.

Let $V$ be a finite dimensional $F$-vector space,
and let $n=\mr{dim}_F(V)$.
Also, let $L$ and $M$ be two full $R$-lattices in $V$.
We will define the invariant exponents $s_R(L,M)$
of $M$ relative to $L$ over $R$ as a multiset of integers.
We abuse the notation by denoting multisets with the same notation
as for ordered sets, for instance $(s_0,...,s_{n-1})$, possibly adding exponents
to indicate multiplicities, for instance $(s_0^{f_0},..., s_{n-1}^{f_{n-1}})$.
The definition of $s_R(L,M)$ goes as follows. 
There exist a basis $(e_0,...,e_{n-1})$ of $L$ over $R$ and
$s_0,...,s_{n-1}\in\bb{Z}$ such that $(\varrho^{s_0}e_0, ..., \varrho^{s_{n-1}}e_{n-1})$
is a basis of $M$ over $R$.
We define $s_R(L,M)=(s_0,....,s_{n-1})$, where one can prove that the multiset
$(s_0,....,s_{n-1})$ only depends on the pair $(L,M)$ and not on the chosen basis. 

Let $V$ be endowed with the structure of $F$-Lie algebra,
and assume that $[V,V]=V$.
Let $L$ be a full $R$-lattice in $V$.
Then $[L,L]$ is a full $R$-lattice in $V$.
We define the $s$-invariants $s_R(L)$ of $L$ over $R$ to be the invariant
exponents $s_R(L, [L,L])$ of $[L,L]$ relative to $L$ over $R$.

\bigskip

For the third part of the section, we keep the assumptions of the second part,
and we introduce some notation for the classification of maximal proper submodules.
Let $L$ be an $R$-lattice, and let $(x_\lambda)_{0\les \lambda < n}$ be a basis
of $L$ over $R$, where $n=\mr{dim}_R(L)$.
We consider pairs $(\mu, b)$ 
where $\mu$ is an integer with $0\les \mu<n$
and $b=(b_\lambda)_{\lambda\neq\mu}$ is a sequence of elements of $R$.
With any such pair we associate a maximal proper $R$-submodule $N_{\mu,b}$ of $L$ defined
to be the span over $R$ of $(y_\lambda)$, where
\begin{eqnarray*}
y_\mu 
&=& 
\varrho x_\mu 
\\
y_\lambda 
&=& 
x_\lambda + b_\lambda x_\mu\quad\mbox{if }\lambda\neq \mu . 
\end{eqnarray*}
Any maximal proper $R$-submodule of $L$ is equal to $N_{\mu,b}$ for some $(\mu,b)$.
The proof of the following lemma is left to the reader.

\begin{lemma}
\label{ldvr1}
Consider pairs $(\mu, b)$ and $(\nu, c)$.
The following holds.
\begin{enumerate}
\item 
Assume that $\mu =\nu$.
Then $N_{\mu,b}=N_{\mu, c}$ if and only if $b_\lambda\equiv_\varrho c_\lambda$ 
for all $\lambda\neq \mu$. 

\item 
Assume that $\mu\neq \nu$. 
Then $N_{\mu,b}=N_{\nu, c}$ if and only if
$b_\nu\not\equiv_\varrho 0$, $c_\mu \equiv_\varrho b_\nu^{-1}$,
and $c_\lambda \equiv_\varrho -b_\nu^{-1}b_\lambda$ for all $\lambda\neq \mu,\nu$.
\end{enumerate}
\end{lemma}

%%%%%%%%%%%%%%%%%%%%
\subsection{Lemmas over the ring of integers of a $p$-adic field}
\label{sedvr}

\noindent
Let $K$ be a finite extension of $\bb{Q}_p$, and let $\cl{O}_K$ be the integral closure
of $\bb{Z}_p$ in $K$. Let $d$, $e$, and $f$ be the degree, ramification index, and inertia degree
of $K/\bb{Q}_p$, respectively. Let $\pi\in\cl{O}_K$ be a uniformizing parameter.

\begin{remark}
\label{redvr1}
Let $B$ be maximal proper $\bb{Z}_p$-submodule of $\cl{O}_K$. 
Since $[\cl{O}_K:B]=p$ and $[\cl{O}_K: \pi\cl{O}_K]=p^f$,
we see that the condition $B\subseteq \pi\cl{O}_K$ 
implies $f=1$.
Hence,
the conditions
$B\subseteq \pi\cl{O}_K$
and 
$B=\pi\cl{O}_K$ are equivalent.
\end{remark}

\medskip 

We recall the following known property; see, for instance, \cite[Theorem 5.6, page 71]{ReiMaxOrd}.

\begin{lemma}
\label{ledvr1}
Let $(\alpha_i)_{0\les i <f}$ be a sequence of elements of $\cl{O}_K$ whose reduction modulo 
$\pi$ is a basis of $\cl{O}_K/\pi\cl{O}_K$ over $\bb{F}_p$. 
Let $(\beta_j)_{0\les j < e}$ be a sequence of elements of $\cl{O}_K$ such that
$v_K(\beta_j)= j$ for all $j$,
where $v_K:K\rar \bb{Z}\cup\{\infty\}$ is the valuation of $K$.
Then $(\alpha_i\beta_j)_{0\les i < f,\,0\les j < e}$
is a basis of $\cl{O}_K$ over $\bb{Z}_p$.
\end{lemma}

\medskip 

We fix a sequence $(\alpha_i)_{0\les i <f}$ of elements of $\cl{O}_K$ whose reduction modulo 
$\pi$ is a basis of $\cl{O}_K/\pi\cl{O}_K$ over $\bb{F}_p$. 
Moreover, we assume that $\alpha_0=1$.
We define $\varepsilon_{ij} = \alpha_i\pi^j$ for $0\les i < f$ and $0\les j <e$,
so that $(\varepsilon_{ij})$ is a basis of $\cl{O}_K$ over $\bb{Z}_p$. 
Observe that $\varepsilon_{00}=1$.
The reader should recall the definition of $s$-invariants
from Section \ref{sdvr}.

\begin{lemma}
\label{ldvr3}
Let $s\in\bb{Z}$, and let $s= qe+r$, where $q,r\in\bb{Z}$ and $0\les r <e$.
Then $s_{\bb{Z}_p}(\cl{O}_K, \pi^s\cl{O}_K) = (q^{d-rf},(q+1)^{rf})$.
\end{lemma}

\begin{proof}
Since $v_{K}(p)=e$, by Lemma \ref{ledvr1} we have that 
$(\alpha_i p^q\pi^{j-s}:\,0\les i< f,\,r\les j<e; \,\alpha_i p^{q+1}\pi^{j-s}:\,0\les i< f, \,0\les j<r)$
is a basis of $\cl{O}_K$ over $\bb{Z}_p$.
By multiplying these basis elements by $\pi^s$, we get that 
$(p^q\alpha_i \pi^{j}:\,0\les i< f,\,r\les j<e; \, p^{q+1}\alpha_i \pi^{j}:\,0\les i< f, \,0\les j<r)$
is a basis of $\pi^s\cl{O}_K$ over $\bb{Z}_p$.
The result follows.
\end{proof}

\begin{remark}
\label{rdvr3}
For $0\les i< f$ and $0\les j<e$,
define 
$$
\gamma_{ij} =
\left\{
\begin{array}{ll}
p\vep_{ij}
& 
\mbox{if}\quad j=0
\\
\vep_{ij}
& 
\mbox{if}\quad  j\neq 0. 
\end{array}
\right.
$$
By applying the argument in the proof of Lemma \ref{ldvr3} with $s=1$, it follows that $(\gamma_{ij})$
is a basis of $\pi\cl{O}_K$ over $\bb{Z}_p$.
\end{remark}

\begin{corollary}
\label{cdvr3}
Let $V$ be a finite dimensional $K$-vector space, and let $n=\mr{dim}_K(V)$.
Let $L$ and $M$ be full $\cl{O}_K$-lattices in $V$, and let 
$s_{\cl{O}_K}(L,M)=(s_0,...,s_{n-1})$.
For $0\les l < n$, let $s_l = q_le+r_l$, where $q_l,r_l\in\bb{Z}$ and $0\les r_l <e$.
Then $s_{\bb{Z}_p}(L,M) = (q_0^{d-r_0f},(q_0+1)^{r_0f}, ..., q_{n-1}^{d-r_{n-1}f},(q_{n-1}+1)^{r_{n-1}f})$.
\end{corollary}

\begin{proof}
There exist a basis $e_0,..., e_{n-1}$ of $L$ over $\cl{O}_K$ and integers $s_0,...,s_{n-1}$ such that
$(\pi^{s_l}e_l)$ is a basis of $M$ over $\cl{O}_K$. By definition, $S_{\cl{O}_K}(L,M)=(s_0,....,s_{n-1})$.
Now the corollary follows from Lemma \ref{ldvr3}.
\end{proof}

\begin{lemma}
\label{ldvr2}
Let $B$ be a maximal proper $\bb{Z}_p$-submodule of $\cl{O}_K$.
Then the following holds. 
\begin{enumerate}
\item 
If $\pi\cl{O}_K\subseteq B$ then $s_{\bb{Z}_p}(B, \pi\cl{O}_K) = (0^{d-f+1}, 1^{f-1})$.

\item 
If $\pi\cl{O}_K\not\subseteq B$ then $s_{\bb{Z}_p}(B, \pi\cl{O}_K) = (-1, 0^{d-f-1}, 1^{f})$.
\end{enumerate}
\end{lemma}

\begin{proof2}
We will show that the relative invariant exponents $s_{\bb{Z}_p}(B, \pi\cl{O}_K)$ are either equal to
$(0^{d-f+1}, 1^{f-1})$ or to $(-1, 0^{d-f-1}, 1^{f})$.
Since the invariant exponents are all nonnegative if and only if $\pi\cl{O}_K\subseteq B$, the result follows.
We use the letters $\lambda$ and $\mu$ for pairs of integers such as 
$\lambda = (\lambda_0,\lambda_1)$,
where $0\les \lambda_0<f$ and $0\les \lambda_1 <e$.
Recall the basis $(\vep_\lambda)$ of $\cl{O}_K$ over $\bb{Z}_p$ introduced above,
and consider the basis $(\gamma_\lambda)$ of $\pi\cl{O}_K$ over $\bb{Z}_p$ given in
Remark \ref{rdvr3}.
There exist $\mu$ and a sequence $(b_\lambda)_{\lambda \neq \mu}$ of elements of $\bb{Z}_p$
such that $(\beta_\lambda)$ is a basis of $B$ over $\bb{Z}_p$, where
$$
\beta_\lambda =
\left\{
\begin{array}{ll}

p\vep_\mu
& 
\mbox{if}\quad  \lambda=\mu
\\

\vep_\lambda +b_\lambda \vep_\mu 
& 
\mbox{if}\quad \lambda\neq \mu. 
\\
\end{array}
\right.
$$
The idea of the proof is to write the basis $(\gamma_\lambda)$
in coordinates with respect to $(\beta_\lambda)$ over $\bb{Q}_p$,
and then compute the Smith canonical form of the resulting 
change-of-basis matrix. 
We divide the proof into two cases.
\begin{enumerate}
\item 
Assume that $\mu_1 = 0$
and, for all $\lambda$, if $\lambda\neq \mu$ and $\lambda_1\neq 0$ then $b_\lambda\equiv_p 0$.
By defining 
$$
\delta_\lambda =
\left\{
\begin{array}{ll}
\gamma_\mu 
& 
\mbox{if}\quad \lambda = \mu 
\\

\gamma_\lambda + b_\lambda \gamma_\mu 
& 
\mbox{if}\quad \lambda\neq \mu \mbox{ and }\lambda_1=0 
\\

\gamma_\lambda + p^{-1}b_\lambda \gamma_\mu 
& 
\mbox{if}\quad \lambda\neq \mu \mbox{ and }\lambda_1\neq 0,
\\
\end{array}
\right.
$$
one proves that $(\delta_\lambda)$ is a basis of $\pi\cl{O}_K$ over $\bb{Z}_p$
and 
$$
\delta_\lambda =
\left\{
\begin{array}{ll}
\beta_\mu
& 
\mbox{if}\quad \lambda = \mu 
\\

p\beta_\lambda
& 
\mbox{if}\quad \lambda\neq\mu \mbox{ and }\lambda_1=0 
\\

\beta_\lambda
& 
\mbox{if}\quad \lambda\neq \mu \mbox{ and }\lambda_1\neq 0, 
\\
\end{array}
\right.
$$
from which the invariant exponents are computed
to be $(0^{d-f+1}, 1^{f-1})$.

\item 
Assume that $\mu_1\neq 0$ or there exists $\lambda$ such that
$\lambda\neq \mu$, $\lambda_1\neq 0$, and $b_\lambda\not\equiv_p 0$. 
The proof is similar to the one of the first case after defining
$$
\delta_\lambda =
\left\{
\begin{array}{ll}
\gamma_\mu 
& 
\mbox{if}\quad \lambda = \mu 
\\

\gamma_\lambda + pb_\lambda \gamma_\mu 
& 
\mbox{if}\quad \lambda\neq\mu \mbox{ and }\lambda_1=0 
\\

\gamma_\lambda + b_\lambda \gamma_\mu 
& 
\mbox{if}\quad \lambda\neq\mu \mbox{ and }\lambda_1\neq 0,
\\
\end{array}
\right.
$$
and the invariant exponents are $(-1, 0^{d-f-1}, 1^{f})$.\ep 
\end{enumerate}
\end{proof2}

\bigskip

In the following lemma, the indices take values $0\les \lambda_0,\mu_0<f$
and 
$0\les \lambda_1,\mu_1<e$.

\begin{lemma}
\label{ledvr2}
Given $(\mu_0,\mu_1)$ and a sequence 
$(b_{\lambda_0\lambda_1})_{(\lambda_0,\lambda_1)\neq (\mu_0,\mu_1)}$
of elements of $\bb{Z}_p$,
define
$\beta_{\mu_0\mu_1}
=p\varepsilon_{\mu_0\mu_1}$
and 
$\beta_{\lambda_0\lambda_1}
= 
\varepsilon_{\lambda_0\lambda_1} +b_{\lambda_0\lambda_1} \varepsilon_{\mu_0\mu_1}$ 
if  
$(\lambda_0,\lambda_1)\neq (\mu_0,\mu_1)$.
Let $B$ be the span of $(\beta_{\lambda_0\lambda_1})$ over $\bb{Z}_p$.
Then the following are equivalent.
\begin{enumerate}
\item 
\label{ledvr2.1}
$B\subseteq \pi\cl{O}_K$.
\item 
\label{ledvr2.2}
$f=1$, $\mu_1 =0$, and 
for all $(\lambda_0,\lambda_1)$ with $(\lambda_0,\lambda_1)\neq (\mu_0,\mu_1)$
and $\lambda_1\neq 0$ we have $b_{\lambda_0\lambda_1}\equiv_p 0$.
\end{enumerate}
\end{lemma}

\begin{proof}
Observe that when $f=1$ the index $\lambda_0$ only takes the value $0$.
We first prove that (\ref{ledvr2.2}) implies (\ref{ledvr2.1}).
Assume (\ref{ledvr2.2}).
Since $f=1$, we can write a generic element $\gamma$ of $B$ as
$\sum_{\lambda_1} c_{\lambda_1}\beta_{0\lambda_1}$
for some sequence $(c_{\lambda_1})$ of elements of $\bb{Z}_p$. 
Since $\mu_1=0$, 
we have
$$
\gamma 
=
 c_0\beta_{00} +
\sum_{\lambda_1\neq 0}c_{\lambda_1}\beta_{0\lambda_1}
=
c_0p+
\sum_{\lambda_1\neq 0}c_{\lambda_1} (\pi^{\lambda_1}+b_{0\lambda_1}).
$$
Since $b_{0\lambda_1}\equiv_p 0$ for $\lambda_1\neq 0$
and since $\pi$ divides $p$ in $\cl{O}_K$, we see that $\gamma$
is a multiple of $\pi$ in $\cl{O}_K$.

We now prove that (\ref{ledvr2.1}) implies (\ref{ledvr2.2}) by contrapositive.
By Remark \ref{redvr1}, if $f\neq 1$ we have $B\not\subseteq \pi\cl{O}_K$.
Assume that $f=1$. If $\mu_1\neq 0$ then $\beta_{00}=1+b_{00}\pi^{\mu_1}$
is in $B$ but not in $\pi\cl{O}_K$. 
If there exists $(\lambda_0,\lambda_1)$ such that $(\lambda_0,\lambda_1)\neq (\mu_0,\mu_1)$, 
$\lambda_1\neq 0$, and $b_{\lambda_0\lambda_1}\not\equiv_p 0$
then, by Lemma \ref{ldvr1}, one can reduce the proof to the case where
$\mu_1\neq 0$.
\end{proof}

\vspace{7mm}

Let $L$ be an $\cl{O}_K$-lattice, and let $n=\mr{dim}_{\cl{O}_K}(L)$.

\begin{remark}
\label{rf2noO_K}
If $M$ is a proper $\cl{O}_K$-submodule of $L$ then $[L:M]\ges p^f$.
It follows that,
for $f\ges 2$, no maximal proper $\bb{Z}_p$-submodule of $L$ is an $\cl{O}_K$-submodule of $L$.
Moreover, let $N$ be a maximal proper $\bb{Z}_p$-submodule of $L$.
If $N$ is an $\cl{O}_K$-submodule of $L$ then $f=1$ and $N$ is a maximal proper $\cl{O}_K$-submodule of $L$.
\end{remark}

We adapt the general notation presented in Section \ref{sdvr} in order to 
treat  the maximal proper $\bb{Z}_p$-submodules 
of $L$. 
We consider indices $\lambda=(\lambda_0,\lambda_1,\lambda_2)$ 
with $0\les \lambda_0<f$, $0\les \lambda_1 <e$, and $0\les \lambda_2 <n$,
and similarly for $\mu$ and $\nu$.
Let $(e_{\lambda_2})$ be a basis of $L$ over $\cl{O}_K$. 
We define $x_\lambda = \varepsilon_{\lambda_0\lambda_1}e_{\lambda_2}$,
so that $(x_\lambda)$ is a basis of $L$ over $\bb{Z}_p$.
Given an index $\mu$ and a sequence $b=(b_\lambda)_{\lambda\neq \mu}$
of elements of $\bb{Z}_p$,
we define 
$$
\begin{array}{rcl}
y_\mu 
&=& 
px_\mu
\\

y_\lambda 
&=& 
x_\lambda +b_\lambda x_\mu \quad\mbox{ if } \lambda \neq \mu,
\end{array}
$$
and 
$N_{\mu,b}$ to be the span  of $(y_\lambda)$ over $\bb{Z}_p$.
Hence, $N_{\mu,b}$ is a maximal proper $\bb{Z}_p$-submodule of $L$,
and any maximal proper $\bb{Z}_p$-submodule of $L$ is equal to $N_{\mu,b}$ for some pair $(\mu, b)$.
Given a pair $(\mu, b)$, we define 
$$
\begin{array}{rcl}
\beta_{\mu_0\mu_1}
&=& 
p\varepsilon_{\mu_0\mu_1}
\\

\beta_{\lambda_0\lambda_1}
&=& 
\varepsilon_{\lambda_0\lambda_1} +b_{\lambda_0\lambda_1\mu_2} \varepsilon_{\mu_0\mu_1} 
\quad\mbox{ if } (\lambda_0,\lambda_1) \neq (\mu_0,\mu_1),
\end{array}
$$
and 
$B_{\mu,b}$ to be the span over $\bb{Z}_p$ of $(\beta_{\lambda_0\lambda_1})$.
Hence, $B_{\mu,b}$ is a maximal proper $\bb{Z}_p$-submodule of $\cl{O}_K$.

\begin{lemma}
\label{ledvr5}
Let $N$ be a maximal proper $\bb{Z}_p$-submodule of $L$.
Consider the following conditions.
\begin{enumerate}
\item 
\label{ledvr5.1}
For all $(\mu, b)$ such that $N=N_{\mu,b}$ we have $B_{\mu,b}\subseteq \pi\cl{O}_K$.
\item 
\label{ledvr5.2}
$f=1$ and for all $(\mu, b)$ such that $N=N_{\mu,b}$
the following holds:
$\mu_1 =0$ and 
for all $\lambda$ with $\lambda_1\neq 0$ we have $b_\lambda\equiv_p 0$.
\item 
\label{ledvr5.3}
$f=1$ and there exists $(\mu, b)$ such that $N=N_{\mu,b}$,
$\mu_1 =0$, and
for all $\lambda$ with $\lambda_1\neq 0$ we have $b_\lambda\equiv_p 0$.
\item 
\label{ledvr5.4}
$N$ is an $\cl{O}_K$-submodule of $L$. 
\end{enumerate}
Then (\ref{ledvr5.1}) $\Rightarrow$ (\ref{ledvr5.2}) $\Rightarrow$ (\ref{ledvr5.3}) $\Rightarrow$ (\ref{ledvr5.4}).
\end{lemma} 

\begin{proof}
Observe that (\ref{ledvr5.2}) $\Rightarrow$ (\ref{ledvr5.3}) is immediate.
We start by proving (\ref{ledvr5.1}) $\Rightarrow$ (\ref{ledvr5.2}).
The claim that $f=1$, under assumption (\ref{ledvr5.1}), follows from Remark \ref{redvr1}.
Assume by contrapositive that there exists $(\mu,b)$ such that $N=N_{\mu,b}$ and either
$\mu_1 \neq 0$ or there exists $\nu$ with $\nu_1\neq 0$ and $b_\nu\not\equiv_p 0$. 
In case $\mu_1\neq 0$, Lemma \ref{ledvr2} implies that 
$B_{\mu,b}\not\subseteq \pi\cl{O}_K$.
In case there exists $\nu_1$ with $\nu_1\neq 0$ and $b_\nu\not\equiv_p 0$, Lemma \ref{ldvr1}
implies that $N=N_{\nu, c}$ for some $c$, and we have 
$B_{\nu,c}\not\subseteq \pi\cl{O}_K$ again by Lemma \ref{ledvr2}.

We now prove that (\ref{ledvr5.3}) $\Rightarrow$ (\ref{ledvr5.4}).
Assume (\ref{ledvr5.3}), and 
observe that, since $f=1$, we have $\lambda_0=0$ for all $\lambda$.
Given $(\mu,b)$ with the properties of assumption (\ref{ledvr5.3}),
we define $z_{\mu_2}=\pi e_{\mu_2}$ and 
$z_{\lambda_2}= e_{\lambda_2}+b_{00\lambda_2}e_{\mu_2}$ for $\lambda_2\neq \mu_2$.
Moreover, we define $M$ to be the span of $(z_{\lambda_2})$ over $\cl{O}_K$,
so that $M$ is a maximal proper $\cl{O}_K$-submodule of $L$.
We will show that $N=M$.
Since $M$ is a proper $\bb{Z}_p$-submodule of $L$, it is enough to show
that the generators $y_\lambda$ of $N$ are in $M$. We have
$$ 
\begin{array}{rcll}
y_\mu &=& pe_{\mu_2}
&
\\
y_\lambda &=& z_{\lambda_2} &\mbox{if }\lambda_1=0\mbox{ and }\lambda_2\neq\mu_2
\\
y_\lambda &=& \pi^{\lambda_1}e_{\lambda_2}+b_{0\lambda_1\lambda_2}e_{\mu_2}&\mbox{if }\lambda_1\neq 0.
\end{array}
$$
By assumption, $\pi$ divides $b_{0\lambda_1\lambda_2}$ for $\lambda_1\neq 0$;
also, observe that $\pi$ divides $p$ in $\cl{O}_K$ and $\pi L\subseteq M$. 
Hence, any generator of $N$ is in $M$, since
it is either a generator of $M$ as well, or a multiple of $\pi$ in $L$.
\end{proof}

\begin{remark}
The conditions in the statment of Lemma \ref{ledvr5} are equivalent.
Since we will not need this fact, we leave the proof to the reader.
\end{remark}

\begin{lemma}
\label{lnotO_KNNLL}
Let $L$ be an $\cl{O}_K$-Lie lattice, and let $N$ be a maximal proper $\bb{Z}_p$-submodule of $L$.
If $N$ is not an $\cl{O}_K$-submodule of $L$ then $[N,N]=[L,L]$.
\end{lemma} 

\begin{proof2}
Assume that $N$ is not an $\cl{O}_K$-submodule of $L$,
and let $(e_{\lambda_2})$ be a basis of $L$ over $\cl{O}_K$.
By Lemma \ref{ledvr5}, there exists $(\mu, b)$ such that $N=N_{\mu,b}$ and 
$B:=B_{\mu,b}\not\subseteq \pi\cl{O}_K$.
We have $[L,L]=\sum_{\lambda_2,\nu_2: \lambda_2\neq \nu_2} \cl{O}_K[e_{\lambda_2},e_{\nu_2}]$.
Since it is clear that $[N,N]\subseteq [L,L]$, it is enough to prove that 
$\cl{O}_K[e_{\lambda_2},e_{\nu_2}]\subseteq [N,N]$ for all $\lambda_2$ and $\nu_2$ with $\lambda_2\neq \nu_2$.
We take $\lambda_2$ and $\nu_2$ with $\lambda_2\neq \nu_2$, and 
we divide the proof into two cases.
\begin{enumerate}
\item 
Assume that $\lambda_2 = \mu_2$ or $\nu_2= \mu_2$. 
Without loss of generality, we can assume that $\lambda_2=\mu_2$.
Let $z\in \cl{O}_K[e_{\mu_2},e_{\nu_2}]$,
and let $\delta\in B\cap \cl{O}_K^*$.
Then there exists $\gamma\in\cl{O}_K$ such that $z = [\delta e_{\mu_2}, \gamma e_{\nu_2}]$.
We have 
$\gamma= \sum_{\nu_0,\nu_1} c_{\nu_0\nu_1} \varepsilon_{\nu_0\nu_1}$
and 
$\delta =\sum_{\lambda_0,\lambda_1} d_{\lambda_0\lambda_1} \beta_{\lambda_0\lambda_1}$
for some indexed sets $(c_{\nu_0\nu_1})$ and $( d_{\lambda_0\lambda_1})$ of elements of $\bb{Z}_p$.
Hence,
$$z
=
\sum_{\lambda_0,\lambda_1,\nu_0,\nu_1} 
d_{\lambda_0\lambda_1}c_{\nu_0\nu_1}
[\beta_{\lambda_0\lambda_1}e_{\mu_2},  \varepsilon_{\nu_0\nu_1}e_{\nu_2}].
$$ 
Since $[\beta_{\lambda_0\lambda_1}e_{\mu_2},  \varepsilon_{\nu_0\nu_1}e_{\nu_2}]=
[\beta_{\lambda_0\lambda_1}e_{\mu_2},  \varepsilon_{\nu_0\nu_1}e_{\nu_2}+b_{\nu_0\nu_1\nu_2} \varepsilon_{\mu_0\mu_1}e_{\mu_2}]
=[y_{\lambda_0\lambda_1\mu_2}, y_{\nu_0\nu_1\nu_2}]$,
we have $z\in [N,N]$.

\item 
Assume that $\lambda_2 \neq \mu_2$ and $\nu_2\neq \mu_2$. 
Let $z\in \cl{O}_K[e_{\lambda_2},e_{\nu_2}]$, so that
$z=\gamma [e_{\lambda_2},e_{\nu_2}]$ for some $\gamma\in\cl{O}_K$. 
We have 
$\gamma= \sum_{\lambda_0,\lambda_1} c_{\lambda_0\lambda_1} \varepsilon_{\lambda_0\lambda_1}$
for some indexed set $(c_{\lambda_0\lambda_1})$ of elements of $\bb{Z}_p$.
Observe that 
$[y_{\lambda_0\lambda_1\lambda_2}, y_{00\nu_2}]= 
\varepsilon_{\lambda_0\lambda_1}[e_{\lambda_2},e_{\nu_2}]
+
b_{00\nu_2}\varepsilon_{\lambda_0\lambda_1}\varepsilon_{\mu_0\mu_1}[e_{\lambda_2},e_{\mu_2}]
+
b_{\lambda_0\lambda_1\lambda_2}\varepsilon_{\mu_0\mu_1}[e_{\mu_2},e_{\nu_2}]
$.
Hence,
\begin{eqnarray*}
\sum_{\lambda_0,\lambda_1}c_{\lambda_0\lambda_1}[y_{\lambda_0\lambda_1\lambda_2}, y_{00\nu_2}]
&=&
\gamma[e_{\lambda_2},e_{\nu_2}]+
b_{00\nu_2}\gamma\varepsilon_{\mu_0\mu_1}[e_{\lambda_2},e_{\mu_2}]+
\\
&&
\left(\sum_{\lambda_0,\lambda_1}c_{\lambda_0\lambda_1}b_{\lambda_0\lambda_1\lambda_2}\right)
\varepsilon_{\mu_0\mu_1}[e_{\mu_2},e_{\nu_2}].
\end{eqnarray*}
By definition, the sum on the left-hand side is in $[N,N]$;
also, by the previous item, the last two terms on the right-hand side are in $[N,N]$.
Since the first term on the right-hand side is $z$, it follows that $z\in[N,N]$.\ep 
\end{enumerate}
\end{proof2}

%%%%%%%%%%%%%%%%%%%%%%%%%

\subsection{Comments on division algebras}
\label{divalg}

By keeping the notation of the context of Theorem \ref{tmainalt} 
which is stated in the Introduction,
we let $D$ be a finite dimensional central division $K$-algebra,
where $K$ is a finite field extension of $\bb{Q}_p$.
We recall that $n$ denotes the index of $D$ over $K$,
$\cl{O}_K$ denotes the integral closure of $\bb{Z}_p$ in $K$, and
$\Delta$ denotes the unique maximal $\cl{O}_K$-order in $D$.
Moreover, we let $\pi\in \cl{O}_K$ be a uniformizing parameter.
We are going review some of the structure properties of $D$
and collect some lemmas for later use.
Specific results for $n=2$ and $n\ges 3$ will be given in Section 
\ref{sn2} and Section \ref{sO_K}, respectively. 

From the general theory
we know that there exist $F$, $\theta$, and $\Pi$ with the following properties; 
see, for instance, \cite[Section 14]{ReiMaxOrd}.

\begin{enumerate}
\item 
$K\subseteq F\subseteq D$ and $F$ is a unital $K$-subalgebra of $D$
and an unramified field extension of $K$ of degree $n$; 
moreover, $F/K$ is Galois and cyclic.

\item 
$\theta$ is a generator of $\mr{Gal}(F/K)$,
$\Pi\in D$, $\Pi^n=\pi$, and $\Pi\alpha=\theta(\alpha)\Pi$ for all 
$\alpha\in F$.

\item 
$D$ is generated by $F$ and $\Pi$ as a unital $K$-algebra.
More precisely,
$$
D = 
\bigoplus_{0\les j <n} F\Pi^j.
$$

\item 
Let $\cl{O}_F$ be the integral closure of $\cl{O}_K$ in $F$.
Then  
$$
\Delta = 
\bigoplus_{0\les j <n} \cl{O}_F\Pi^j.
$$

\item 
Let $t_{D/K}:D\rar K$ be the reduced-trace map, and let $T_{F/K}:F\rar K$ be the trace map.
Let $x =\sum_{j} \alpha_j\Pi^j$, where $\alpha_j\in F$, be an element of $D$.
Then $t_{D/K}(x)=T_{F/K}(\alpha_0)$. 
\end{enumerate}

\noindent 
We denote the residue fields of $\cl{O}_K$ and $\cl{O}_F$
by $\kappa_K=\cl{O}_K/\pi\cl{O}_K$ and $\kappa_F =\cl{O}_F/\pi\cl{O}_F$,
respectively,
and by 
$v_K:K\rar \bb{Z}\cup\{\infty\}$
and $v_F:F\rar \bb{Z}\cup\{\infty\}$
the respective valuations.
We denote by $\mr{pr}:\cl{O}_F\rar \kappa_F$ the projection,
and we adopt the notation $\overline{\alpha}=\mr{pr}(\alpha)$
when convenient. 
The extension $\kappa_F/\kappa_K$ has degree $n$. 
Let $F^0$ be the kernel of the trace map $F\rar K$,
and define $\cl{O}_F^0=\cl{O}_F\cap F^0$. 
Let $\kappa_F^0$ be the kernel of the trace map $\kappa_F\rar\kappa_K$.
Observe that $F^0$ and $\kappa_F^0$ have dimension
$n-1$ over $K$ and $\kappa_K$, respectively.

The following lemma is an immediate consequence of the properties listed above,
and it will be used constantly in Section \ref{sO_K}.
From the lemma we see that $D$ is graded over 
$\bb{Z}/n\bb{Z}$ as a $K$-Lie algebra.

\begin{lemma}
\label{lcomm3}
For all $j,k\in \bb{Z}$ and $\alpha,\beta\in F$
we have 
$
[\alpha\Pi^j, \beta\Pi^k]=(\alpha\theta^j(\beta)-\beta\theta^k(\alpha))\Pi^{j+k}
$,
where $[\,\cdot\, , \cdot\,]$ denotes the commutator in $D$. 
\end{lemma}

Let $sl_1(D)$ be the kernel of the reduced-trace map $D\rar K$,
so that $sl_1(\Delta)=\Delta\cap sl_1(D)$.
We have 
$$
sl_1(D) = 
F^0\oplus \left( 
\bigoplus_{1\les j <n} F\Pi^j
\right)
\qquad\mbox{and}\qquad 
sl_1(\Delta) = 
\cl{O}_F^0\oplus \left( 
\bigoplus_{1\les j <n} \cl{O}_F\Pi^j
\right).
$$

\begin{remark}
\label{rsimple}
Assume that $n\ges 2$.
Since $sl_1(D)$ is simple both as a $K$-Lie algebra and as a $\bb{Q}_p$-Lie algebra,
$sl_1(\Delta)$ is just infinite both
as an $\cl{O}_K$-Lie lattice and as a $\bb{Z}_p$-Lie lattice; see Lemma \ref{lsjipid2}.
\end{remark}

\bigskip 

We point out some properties that will be used in 
Section \ref{sn2} and Section \ref{sO_K}. 

%%%%%%%%%%%%%%%%%%%%%%%%%%%%%%

\begin{remark}
\label{rdecO}
There exists $\tau_0\in\cl{O}_F$ such that $v_K(T_{F/K}(\tau_0))=0$,
otherwise the trace map of the extension $\kappa_F/\kappa_K$ would be identically zero.
We have $\cl{O}_F=\cl{O}_K\tau_0\oplus \cl{O}_F^0$.
In particular, there exists a  
basis $(\tau_i)_{0\les i <n}$
of $\cl{O}_F$ over $\cl{O}_K$ such that
$(\tau_i)_{1\les i <n}$ is a basis of $\cl{O}_F^0$ over $\cl{O}_K$.
\end{remark}

\begin{lemma}
\label{lpiof}
$\pi(\cl{O}_F^0)=\kappa_F^0$.
\end{lemma} 

\begin{proof}
The inclusion `$\subseteq$' is clear.
For `$\supseteq$',
let $(\tau_i)_{0\les i <n}$ be a basis of $\cl{O}_F$ over $\cl{O}_K$
such that $(\tau_i)_{1\les i <n}$ is a basis of $\cl{O}_F^0$ over $\cl{O}_K$.
Observe that  $(\overline{\tau}_i)_{0\les i <n}$ is a basis of $\kappa_F$ over $\kappa_K$
and that $(\overline{\tau}_i)_{1\les i <n}$ is a basis of $\kappa_F^0$ over $\kappa_K$.
Also, note that $T_{F/K}(\tau_0)\not\equiv_\pi 0$.
Let $x\in \kappa_F^0$. There exists $\alpha \in \cl{O}_F$ such that $\mr{pr}(\alpha)=x$.
We have $\alpha = \sum_{0\les i<n} a_i\tau_i$ for some $a_i\in\cl{O}_K$.
We have $\pi \,|\,a_0$, since  $a_0T_{F/K}(\tau_0)=T_{F/K}(\alpha)\equiv_\pi 0$. 
Let $\beta = \alpha-a_0\tau_0$. Then $\beta\in \cl{O}_F^0$ and $\mr{pr}(\beta)=x$. 
\end{proof}

\begin{lemma}
\label{cgenogenf}
Let $U$ be an $\cl{O}_K$-submodule of $\cl{O}_F$.
Then the following holds.
\begin{enumerate}
\item 
If $\mr{pr}(U)=\kappa_F$ then $U=\cl{O}_F$.

\item
If $U\subseteq \cl{O}_F^0$ and $\mr{pr}(U)=\kappa_F^0$
then $U=\cl{O}_F^0$.
\end{enumerate}
\end{lemma}

\begin{proof}
The statement is a consequence of Nakayama's Lemma 
and Lemma \ref{lpiof}.
\end{proof}

%%%%%%%%%%%%%%%%%%
\subsection{The case $n=2$}
\label{sn2}

In the context of Section \ref{divalg}, we 
aim to prove the following theorem.

\begin{theorem}
\label{tn2}
Assume that $n=2$ and $p\neq 2$, and let $m\in\bb{N}$. 
Then $p^m sl_1(\Delta)$, regarded as a $\bb{Z}_p$-Lie lattice,
is not self-similar of index $p$.
\end{theorem}

The theorem is a consequence of Lemma \ref{ln2_1} and Lemma \ref{ln2_2} proven below.
For the expository convenience of this section, we introduce the terminology 
``standard basis" in the following definition, which the
reader should compare with \cite[Lemma 2.20]{NS2019}. 
We point out that, in general, standard bases do not exist, even for 3-dimensional unsolvable Lie lattices.

\begin{definition}
Let $L$ be a 3-dimensional $\cl{O}_K$-Lie lattice.
A standard basis of $L$ is a basis $(e_0, e_1, e_2)$ of $L$ over $\cl{O}_K$
with the property that $[e_i,e_{i+1}]= u_{i+2}\pi^{s_{i+2}}e_{i+2}$ for all $i\in\bb{Z}/3\bb{Z}$,
where $u_i\in \cl{O}_K^*$ for all $i$,
$s_0=1$, $s_1=s_2=0$, and $-u_1u_2$ is not a square modulo $\pi$ in $\cl{O}_K$, that is,
there is no $\alpha\in\cl{O}_K$ such that $-u_1u_2\equiv_\pi \alpha^2$.
\end{definition}

\begin{lemma}
\label{ln2_1}
Assume that $n=2$ and $p\neq 2$.
Then there exists a standard basis of $sl_1(\Delta)$.
\end{lemma}

\begin{proof}
Since $p$ is odd, $2\in \cl{O}_K^*$.
Hence, since $T_{F/K}(1)=2$, there exists $\xi\in\cl{O}_F$ such
that $(1,\xi)$ is a basis of  $\cl{O}_F$ over $\cl{O}_K$
and $T_{F/K}(\xi)= 0$ (see Remark \ref{rdecO}).
By observing that $\theta(\xi)= -\xi$,
so that $N_{F/K}(\xi)=-\xi^2$, we see that $\xi^2\in \cl{O}_K$; moreover, since $\xi\in\cl{O}_F^*$,
we have $\xi^2\in \cl{O}_K^*$. 
By defining $e_0 = \xi$, $e_1=\Pi$, and $e_2= \xi\Pi$,
we get a presentation of $sl_1(\Delta)$ of the form $sl_1(\Delta)=\cl{O}_Ke_0\oplus \cl{O}_Ke_1\oplus \cl{O}_Ke_2$
with $[e_i, e_{i+1}]= u_{i+2}\pi^{s_{i+2}}e_{i+2}$ for $i\in\bb{Z}/3\bb{Z}$, 
where 
$s_0 =1$, $s_1=s_2=0$, $u_0 = -2\pi$, $u_1 = -2\xi^2$, and $u_2= 2$.
Since $-u_1u_2=4\xi^2$, 
it remains to show that $\xi^2$ is not a square modulo $\pi$ in $\cl{O}_K$, which follows
from the fact that the image of $\xi$ in $\kappa_F$ is not in $\kappa_K$.
\end{proof}

\bigskip

For the proof of Lemma \ref{ln2_2} below, we need two lemmas,
for which we introduce the following context.
Let $L$ be a 3-dimensional $\cl{O}_K$-Lie lattice that
admits a standard basis $(e_0, e_1, e_2)$.
We define 
$M_0 = [L,L]=\mr{Span}_{\cl{O}_K}(\pi e_0,e_1,e_2)$,
which is a maximal proper $\cl{O}_K$-submodule of $L$,
and advise the reader to recall the definition of $s$-invariants
from Section \ref{sdvr}.
The next lemma is a version of \cite[Lemma 2.25, Lemma 2.26]{NS2019} 
where the ring of coefficients is $\cl{O}_K$ in place of $\bb{Z}_p$,
and where $M$ is not assumed to be a subalgebra.  
We leave to the reader to check that the new lemma 
has the same proof as its original counterparts.

\begin{lemma}
\label{ln2sinvM2}
Let $M$ be a maximal proper $\cl{O}_K$-submodule of $L$.
Then the following holds.
\begin{enumerate}
\item 
If $M=M_0$ then $s_{\cl{O}_K}(M)= (0, 1^2)$ and $M=[L,L]$. 

\item 
If $M\neq M_0$ then $s_{\cl{O}_K}(M)= (-1,1,2)$ and $M+[M,M]=L$.
\end{enumerate}
\end{lemma}

\begin{lemma}
\label{ln2sinvN}
Let $N$ be a maximal proper $\bb{Z}_p$-submodule of $L$.
Then the following holds.
\begin{enumerate}
\item 
If $N$ is not an $\cl{O}_K$-submodule of $L$ then:
 \begin{enumerate}
 \item 
 if $[N,N]\subseteq N$ then $s_{\bb{Z}_p}(N)=(0^{3d-f+1}, 1^{f-1})$;
 \item 
 if $[N,N]\not\subseteq N$ then $s_{\bb{Z}_p}(N)=(-1, 0^{3d-f-1}, 1^{f})$.
 \end{enumerate} 

\item 
If $N=M_0$ 
then $s_{\bb{Z}_p}(N)=(0^{3d-2f},1^{2f})=(0^{3d-2},1^{2})$.

\item 
If $N$ is an $\cl{O}_K$-submodule of $L$ and $N\neq M_0$ 
then: 
 \begin{enumerate}
 \item 
 if $d=1$ then $s_{\bb{Z}_p}(N)= (-1, 1, 2)$;
 \item 
 if $d\ges 2$ then $s_{\bb{Z}_p}(N)=(-1^{f}, 0^{3d-4f}, 1^{3f})= (-1, 0^{3d-4}, 1^{3})$.
 \end{enumerate} 
\end{enumerate}
\end{lemma}

\begin{proof2}
When $N$ is an  $\cl{O}_K$-submodule of $L$, the result follows from a straightforward computation
by applying Lemma \ref{ln2sinvM2} and Corollary \ref{cdvr3};
one has also to observe that we have $f=1$ in this case (see Remark \ref{rf2noO_K}).
Assume that $N$ is not an $\cl{O}_K$-submodule of $L$.
We will show that the $s$-invariants $s_{\bb{Z}_p}(N)=s_{\bb{Z}_p}(N, [N, N])$ are
either $(0^{3d-f+1}, 1^{f-1})$ or $(-1, 0^{3d-f-1}, 1^{f})$.
Since the $s$-invariants are all nonnegative if and only if $[N,N]\subseteq N$, the result follows.
Let $(e_0,e_1,e_2)$ be a standard basis of $L$, and  recall that 
$[L,L]=\cl{O}_K\pi e_0\oplus \cl{O}_Ke_1\oplus \cl{O}_Ke_2$.
By Lemma \ref{lnotO_KNNLL}, $[N,N]=[L,L]$.
We use the notation
of Section \ref{sedvr}; in particular, 
from the paragraph below Remark \ref{rf2noO_K},
the reader should recall the domains where the indices 
$\lambda=(\lambda_0,\lambda_1,\lambda_2)$ and the pairs $(\mu,b)$ take values,
and 
the definition of $x_\lambda$, $y_\lambda$, $N_{\mu,b}$, and $B_{\mu,b}$.
We divide the proof into two cases.

\begin{enumerate}
\item 
Assume that there exists $(\mu,b)$ such that $N=N_{\mu,b}$ and $\mu_2\neq 0$.
We define 
$$
z_\lambda =
\left\{
\begin{array}{rl}
x_\lambda 
&\quad \mbox{if } \lambda_2\neq 0
\\ 

x_\lambda 
&\quad \mbox{if } \lambda_2 = 0\mbox{ and }\lambda_1\neq 0
\\ 

px_\lambda 
&\quad \mbox{if } \lambda_2 = 0\mbox{ and }\lambda_1 = 0,
\end{array}
\right.$$
and from Remark \ref{rdvr3} one can see that $(z_\lambda)$ is a basis of $[L,L]$ over $\bb{Z}_p$.
We define a new basis $(w_\lambda)$ of $[L,L]$ over $\bb{Z}_p$ by

$$
w_\lambda =
\left\{
\begin{array}{rl}
z_\mu 
&\quad \mbox{if } \lambda = \mu
\\ 

z_\lambda+b_\lambda z_\mu 
&\quad \mbox{if } \lambda \neq \mu \mbox{ and } \lambda_2\neq 0
\\ 

z_\lambda+b_\lambda z_\mu 
&\quad \mbox{if } \lambda \neq \mu \mbox{ and } \lambda_2= 0\mbox{ and } \lambda_1\neq 0
\\ 

z_\lambda+pb_\lambda z_\mu 
&\quad \mbox{if } \lambda \neq \mu \mbox{ and } \lambda_2= 0\mbox{ and } \lambda_1 = 0.
\end{array}
\right.$$
One can compute that
$$
w_\lambda =
\left\{
\begin{array}{rl}
p^{-1}y_\mu
&\quad \mbox{if } \lambda = \mu
\\ 

y_\lambda
&\quad \mbox{if } \lambda \neq \mu \mbox{ and } \lambda_2\neq 0
\\ 

y_\lambda
&\quad \mbox{if } \lambda \neq \mu \mbox{ and } \lambda_2= 0\mbox{ and } \lambda_1\neq 0
\\ 

py_\lambda
&\quad \mbox{if } \lambda \neq \mu \mbox{ and } \lambda_2= 0\mbox{ and } \lambda_1 = 0,
\end{array}
\right.$$
from which 
the $s$-invariants are  computed to be $(-1, 0^{3d-f-1}, 1^{f})$.

\item 
Assume that for all $(\mu,b)$ with $N=N_{\mu,b}$ we have $\mu_2= 0$.
We take a pair $(\mu, b)$ with $N=N_{\mu,b}$.
For all $\lambda$ with $\lambda_2\neq 0$, we have $b_\lambda\equiv_p 0$,
otherwise, we could apply Lemma \ref{ldvr1} (with $R=\bb{Z}_p$ and $\varrho = p$) 
to get a pair $(\nu, c)$ such that
$N=N_{\nu, c}$ and $\nu_2\neq 0$.
We claim that $N=Be_0\oplus \cl{O}_Ke_1\oplus \cl{O}_Ke_2$, where $B=B_{\mu, b}$ 
(the claim is proven below). 
It follows that $s_{\bb{Z}_p}(N)=s_{\bb{Z}_p}(N, [L,L]) = S_{\bb{Z}_p}(B,\pi \cl{O}_K) \cup (0^{2d})$.
Lemma \ref{ldvr2} gives the value of $S_{\bb{Z}_p}(B,\pi \cl{O}_K)$, and
the $s$-invariants of $N$ turn out to be
$(0^{3d-f+1}, 1^{f-1})$ or $(-1, 0^{3d-f-1}, 1^{f})$, as desired.
For the proof of the claim, let $z\in Be_0\oplus \cl{O}_Ke_1\oplus \cl{O}_Ke_2$.
Then, for some sequence $(c_\lambda)$ of elements of $\bb{Z}_p$,
we have
\begin{eqnarray*}
z
&=& 
\left(\sum_{\lambda: \lambda_2=0} c_\lambda \beta_{\lambda_0\lambda_1}\right)e_0+
\left(\sum_{\lambda: \lambda_2\neq 0} c_\lambda \varepsilon_{\lambda_0\lambda_1}\right)e_{\lambda_2}
\\
&=&
\sum_{\lambda: \lambda_2=0} c_\lambda y_\lambda +
\sum_{\lambda: \lambda_2\neq 0} c_\lambda (y_\lambda-p^{-1}b_\lambda y_\mu),
\end{eqnarray*}
where $\varepsilon_{\lambda_0\lambda_1}$ and $ \beta_{\lambda_0\lambda_1}$
are defined in Section \ref{sedvr}.
Since $b_\lambda\equiv_p 0$ when $\lambda_2\neq 0$, we see that $z\in N_{\mu,b}=N$.
Hence, $Be_0\oplus \cl{O}_Ke_1\oplus \cl{O}_Ke_2\subseteq N$. 
Since $Be_0\oplus \cl{O}_Ke_1\oplus \cl{O}_Ke_2$ is a maximal proper $\bb{Z}_p$-submodule of $L$ 
and $N$ is proper, we get the equality.\ep 
\end{enumerate}
\end{proof2}

\bigskip 

Now, we are ready to prove the lemma that completes the proof of Theorem \ref{tn2}.

\begin{lemma}
\label{ln2_2}
Let $L$ be a 3-dimensional $\cl{O}_K$-Lie lattice that
admits a standard basis, and let $m\in\bb{N}$. 
Then $p^m L$, regarded as a $\bb{Z}_p$-Lie lattice,
is not self-similar of index $p$.
\end{lemma}

\begin{proof}
Let $\varphi:N\rar p^mL$ be a virtual endomorphism of $p^mL$ of index $p$.
We have to show that $\varphi$ is not simple.
We start by proving that  $p^mL$ is just infinite
as $\bb{Z}_p$-Lie lattice.
With the same proof as for \cite[Lemma 9]{NS2022JGT},
one can see that $L$ is just infinite as an $\cl{O}_K$-Lie lattice.
Then, being nonabelian, the $K$-Lie algebra $L\otimes_{\cl{O}_K} K$ is simple
(Lemma \ref{lsjipid2}).
By \cite[I.6.10]{BouLieGrAlg1}, 
$L\otimes_{\cl{O}_K} K$ is simple as a
$\bb{Q}_p$-Lie algebra.
Observe that $L$ is a full $\bb{Z}_p$-lattice 
in the $\bb{Q}_p$-vector space $L\otimes_{\cl{O}_K} K$,
so that  we can identify the $\bb{Q}_p$-Lie algebras
$L\otimes_{\cl{O}_K} K$ and $L\otimes_{\bb{Z}_p}\bb{Q}_p$.
Hence, by a second application of Lemma \ref{lsjipid2}, $L$ is just infinite as a $\bb{Z}_p$-Lie lattice.
Finally, again by Lemma \ref{lsjipid2}, 
$p^mL$ is just infinite as a $\bb{Z}_p$-Lie lattice, as desired.

Now, by Lemma \ref{lsiminj2}, we can assume that $\varphi$ is injective. 
We define $N'=\varphi(N)$, so that $\varphi$ induces an isomorphism
$N\simeq N'$ of $\bb{Z}_p$-Lie algebras, from which we get $s_{\bb{Z}_p}(N)=s_{\bb{Z}_p}(N')$. 
By index stability 
\cite[Corollary 8]{NSindexstable2021}, the index of $N'$ in $p^mL$ is $p$.
It follows that $p^{-m}N$ and $p^{-m}N'$ are maximal proper $\bb{Z}_p$-submodules
of $L$. Since $N$ and $N'$ have the same $s$-invariants, since
the lists of $s$-invariants that appear in Lemma \ref{ln2sinvN} are
different among the (sub)items 
(by looking at the multiplicities of $1$ and $2$, for instance), 
and since the items of the lemma exhaust all the cases,
we see that one of the following possibilities holds:
\begin{itemize}
\item 
none of $p^{-m}N$ and $p^{-m}N'$ is an $\cl{O}_K$-submodule of $L$;
\item 
$p^{-m}N=M_0=p^{-m}N'$;
\item 
both $p^{-m}N$ and $p^{-m}N'$ are $\cl{O}_K$-submodules of $L$ 
different from $M_0$.
\end{itemize}
By Lemma \ref{lnotO_KNNLL} and Lemma \ref{ln2sinvM2}, one of the following possibilities holds:
\begin{itemize}
\item 
$[p^{-m}N, p^{-m}N] = [L,L] = [p^{-m}N', p^{-m}N']$;

\item 
$p^{-m}N= [L,L] = p^{-m}N'$;

\item 
$p^{-m}N + [p^{-m}N, p^{-m}N] = L = p^{-m}N'+ [p^{-m}N', p^{-m}N']$. 
\end{itemize}
Hence, one of the following possibilities holds:
\begin{itemize}
\item 
$[N, N] = [p^{m}L,p^{m}L] = [N', N']$;

\item 
$p^mN=[p^mL,p^mL]=p^mN'$;

\item 
$p^{m}N + [N, N] = p^{m}(p^{m}L) = p^{m}N'+ [N', N']$. 
\end{itemize}
Finally, one of $[p^{m}L,p^{m}L]$ and $\,p^{m}(p^{m}L)$
is a nonzero ideal of $p^mL$ that is $\varphi$-invariant. 
Hence, $\varphi$ is not simple.
\end{proof}

%%%%%%%%%%%%%%%%%%
\subsection{The case $n\ges 3$}
\label{sO_K}

In the context of Section \ref{divalg}, we will work over $\cl{O}_K$ and aim to prove the following theorem. 

\begin{theorem}
\label{tMMLLO_K}
Assume that $n\ges 3$ and $(p,n)\neq (3,3)$. 
Let $M$ be a maximal proper $\cl{O}_K$-submodule of $sl_1(\Delta)$.
Then $[M,M]=[sl_1(\Delta),sl_1(\Delta)]$.
\end{theorem}

The theorem follows from 
Lemma \ref{lMall}, Corollary \ref{cl01}, and Corollary \ref{cl2n}.
The corollaries themselves are a consequence of a rather long
list of technical lemmas.
We do not make the overall assumption that $n\ges 3$.
Notably, Lemma \ref{lLL} works for $n=2$ and $p\neq 2$ 
as well, and it provides a computation
of $[sl_1(\Delta),sl_1(\Delta)]$ alternative to the one given in
\cite{NS2019},  for instance.
On the other hand, the reader may assume that $n\ges 3$ from the beginning
and not loose anything fundamental.

Recall the realization of $sl_1(\Delta)$ given in Section \ref{divalg}.
We denote
$L=sl_1(\Delta)$, $L_0=\cl{O}_F^0$, 
and $L_j= \cl{O}_F\Pi^j$ for $0<j<n$.
Hence,
$$
L = \bigoplus_{0\les j<n}L_j,
$$
and we observe that $L$ is graded over $\bb{Z}/n\bb{Z}$
as an $\cl{O}_K$-Lie algebra.
We define $\Lambda$ to be the set of pairs $(i,j)$ of integers 
such that $0\les i,j <n$ and $(i,j)\neq (0,0)$. We endow $\Lambda$ with the binary relation
$$
(i,j)\les (l,k)
\quad
:\iff
\quad 
j<k \mbox{ or }(j=k\mbox{ and } i\les l),
$$
which is a linear order.
If we denote an element of $\Lambda$ by $\lambda$, say, then we
denote its components by $(\lambda_0,\lambda_1)$.
For $\lambda\in\Lambda$, we define $\Lambda_\lambda=\{\eta\in\Lambda: \eta < \lambda\}$.
Recall that we denote by $\mr{pr}:\cl{O}_F\rar \kappa_F$ the projection,
and that we adopt the notation $\overline{\alpha}=\mr{pr}(\alpha)$
when convenient. 

\begin{lemma}
\label{lspcbasis}
There exists a basis $(\tau_i)_{0\les i <n}$
of $\cl{O}_F$ over $\cl{O}_K$ such that
the following conditions hold.
\begin{enumerate}
\item 
$(\tau_i)_{1\les i <n}$ is a basis of $\cl{O}_F^0$ over $\cl{O}_K$.

\item 
If $p\nmid n$ then $\tau_0=1$.

\item 
If $p\,|\,n$ then $\overline{\tau}_1 = 1$.

\item 
If $n\ges 2$ and $(p,n)\neq (2,2)$ then $\overline{\tau}_{n-1}\not\in \kappa_K$.
\end{enumerate}
\end{lemma}

\begin{proof}
Let $(t_i)_{0\les i < n}$ be a basis of $\kappa_F$ over $\kappa_K$
with the properties stated in Lemma \ref{lbascyc}.
From Nakayama's Lemma and Lemma \ref{lpiof},
it follows that there exists a basis
$(\tau_i)_{0\les i < n}$ of $\cl{O}_F$ over $\cl{O}_K$
such that 
$(\tau_i)_{1\les i < n}$ is a basis of $\cl{O}_F^0$ over $\cl{O}_K$,
$\tau_0=1$ if $p\nmid n$, and $\overline{\tau}_i=t_i$ for $0\les i <n$.
The other required properties are immediate.
\end{proof}

\bigskip

We fix a basis $(\tau_i)_{0\les i<n}$  of 
$\cl{O}_F$ over $\cl{O}_K$ with the properties
stated in Lemma \ref{lspcbasis}.
For $\eta\in\Lambda$, we define $x_\eta = \tau_{\eta_0}\Pi^{\eta_1}$.
Then $(x_\eta)_{\eta\in\Lambda}$ is a basis of $L$ over $\cl{O}_K$.
For $\lambda\in\Lambda$, $e:\Lambda_\lambda\rar \cl{O}_K$, and
$\eta\in\Lambda$ we define
$$
y_\eta = \left\{
\begin{array}{ll}
x_\eta +e_\eta x_\lambda & \mbox{ if }\eta<\lambda\\
\pi x_\lambda & \mbox{ if }\eta = \lambda\\
x_\eta & \mbox{ if }\eta > \lambda,\\
\end{array}
\right.
$$
and $M_{\lambda, e}=\mr{span}_{\cl{O}_K}(y_\eta:\,\eta\in \Lambda)$,
which is a maximal proper $\cl{O}_K$-submodule of $L$.
For the next lemma see, for instance, \cite[Section 5.2]{AWalgmod}.

\begin{lemma}
\label{lMall}
Let $M$ be a maximal proper $\cl{O}_K$-submodule of $sl_1(\Delta)$. Then there exist
$\lambda\in\Lambda$ and $e:\Lambda_\lambda\rar \cl{O}_K$
such that $M=M_{\lambda,e}$. 
\end{lemma}

\begin{remark}
\label{rdownup}
In most proofs of the statements below we construct 
an $\cl{O}_K$-submodule $U$ of $\cl{O}_F$
(respectively, of $\cl{O}_F^0$), usually through a set of generators,
say $(\alpha_i)_{i\in I}$, where $I$ is a set of indices.
In these situations, the aim is to prove that $U=\cl{O}_F$
(respectively, $U=\cl{O}_F^0$). This is achieved by proving that
$\mr{pr}(U)=\kappa_F$ (respectively, $\mr{pr}(U)=\kappa_F^0$), and by
applying Corollary \ref{cgenogenf}.
Observe that $\mr{pr}(U)$ is generated by $(\overline{\alpha}_i)_{i\in I}$
over $\kappa_K$. 
The proof of $\mr{pr}(U)=\kappa_F$ (respectively, $\mr{pr}(U)=\kappa_F^0$) is achieved by applying the
relevant results of Section \ref{sliecyc} to the residue field extension
$\kappa_F/\kappa_K$ and the generator $\vartheta$ given by the reduction modulo
$\pi$ of the restriction of $\theta$ to $\cl{O}_F$.
We observe that $(\overline{\tau}_i)_{0\les i <n}$ 
is a basis of $\kappa_F$ over $\kappa_K$ that satisfies the properties of Lemma
\ref{lbascyc}, and we note that $\mr{ch}(\kappa_K)=p$.
We will make constant use of Lemma \ref{lcomm3} without further mention.
Also, recall that $\Pi^n=\pi$.
\end{remark}

\begin{lemma}
\label{lLL}
Assume that $(p,n)\neq (2,2)$. Then 
$$
[L,L]=\pi L_0\oplus \left(\bigoplus_{1\les j<n}L_j\right).
$$
\end{lemma}

\begin{proof}
Since the lemma is trivial for $n=1$, we assume that $n\ges 2$.
We denote $L'=[L,L]$ and $L'_j=L'\cap L_j$, for $0\les j <n$. 
Observe that $L'_j$ is generated over $\cl{O}_K$ by the brackets 
$[x_\eta,x_\mu]$ where $\eta,\mu\in\Lambda$ and $\eta_1+\mu_1\equiv_n j$;
also, $L'$ is homogeneous, that is, $L'=\bigoplus_{0\les j <n}L'_j$.
Observe that if $\eta_1+\mu_1=0$ then $[x_\eta,x_\mu]=0$. Hence, $L'_0\subseteq \pi L_0$.
It is enough to show that $\pi L_0\subseteq L'_0$ and $L_j\subseteq L'_j$ for $0<j<n$.

We first argue for $L_j$.
Let $0<j<n$. 
The brackets $[x_{i0},x_{lj}]=(\tau_i-\theta^j(\tau_i))\tau_l\Pi^j$ 
belong to $L'_j$ for $1\les i<n$ and $0\les l<n$.
We will show that $N_j:=\mr{span}_{\cl{O}_K}([x_{i0},x_{lj}]: \,1\les i<n\mbox{ and } 0\les l<n)$
is equal to $L_j$, from which $L_j\subseteq L'_j$ follows.
We have $N_j=U_j\Pi^j$, where 
$U_j=
\mr{span}_{\cl{O}_K}([\tau_i-\theta^j(\tau_i)]\tau_l:  \,1\les i<n\mbox{ and } 0\les l<n)$.
Now the desired result follows from Remark \ref{rdownup} and Lemma \ref{ltaneqa}.

Finally we argue for $\pi L_0$.
The brackets $[x_{01},x_{l,n-1}]=\pi (\tau_0\theta(\tau_l)-\tau_l\theta^{-1}(\tau_0))$ 
belong to $L'_0$ for $0\les l<n$.
We will show that $N_0:=\mr{span}_{\cl{O}_K}([x_{01},x_{l,n-1}]: \,0\les l<n)$
is equal to $\pi L_0$, from which $\pi L_0\subseteq L'_0$ follows.
We have $N_0=\pi U_0$, where 
$U_0=
\mr{span}_{\cl{O}_K}(\tau_0\theta(\tau_l)-\tau_l\theta^{-1}(\tau_0):  \,0\les l<n)$.
Now the desired result follows from Remark \ref{rdownup} and Lemma \ref{lcyc0}.
\end{proof}

\begin{remark}
\label{rp2llmm}
Let $M$ be a maximal proper $\cl{O}_K$-submodule of $L$.
Clearly, $[M,M]\subseteq [L,L]$; hence, in order to prove that
$[L,L]= [M,M]$, it is enough to prove that $[L,L]\subseteq [M,M]$.
For  $(p,n)\neq (2,2)$,
it is enough to prove that $\pi L_0\subseteq [M,M]$
and $L_m\subseteq [M,M]$ for $0<m<n$;
see Lemma \ref{lLL}.
Observe that $\pi L\subseteq M$, 
hence, $\pi^2[L,L]\subseteq [M,M]$.
\end{remark}

\medskip 

For the next sequence of lemmas and corollaries, we take 
$\lambda\in\Lambda$ and $e:\Lambda_\lambda\rar \cl{O}_K$, 
and we define $M=M_{\lambda, e}$.
We order the lemmas according to the order in
which they are applied in the corollaries.

\begin{lemma}
\label{lcal2}
Assume that $\lambda_1=0$,
and let $m$ be an integer such that  $2\les m < n$.
Then 
$L_m\subseteq [M,M]$.
\end{lemma}

\begin{proof} 
Observe that $n\ges 3$.
Denote $\eta=(i,1)$ and $\mu=(l,m-1)$,
for $0\les i <n$ and $0\les l<n$.
Observe that $\lambda<\eta$ and $\lambda <\mu$, 
so
that 
$
[y_\eta,y_\mu]
=
[x_\eta,x_\mu]
$,
which belongs to $[M,M]$.
We will show that $N:=\mr{span}_{\cl{O}_K}([x_{i1},x_{l,m-1}]:
\,0\les i<n\mbox{ and } 0\les l<n)$
is equal to $L_m$, from which $L_m\subseteq [M,M]$ follows.
Since
$
[x_{i1},x_{l,m-1}]
=
(\tau_{i}\theta(\tau_{l})-\tau_{l}\theta^{m-1}(\tau_{i}))\Pi^m
$,
we have $N=U\Pi^m$, where 
$U=
\mr{span}_{\cl{O}_K}(\tau_{i}\theta(\tau_{l})-\tau_{l}\theta^{m-1}(\tau_{i}): 
 \,0\les i<n\mbox{ and } 0\les l<n)$.
Now the desired result follows from Remark \ref{rdownup} and Lemma \ref{lcyc2}.
\end{proof}

\begin{lemma}
\label{lcal2b}
Assume that $n\ges 3$, $\lambda_1=0$, and $\lambda_0<n-1$.
Then 
$L_1\subseteq [M,M]$.
\end{lemma}

\begin{proof} 
Denote $\eta=(n-1,0)$ and $\mu=(l,1)$, for $0\les l<n$.
Observe that $\lambda<\eta$ and $\lambda <\mu$, 
so
that 
$
[y_\eta,y_\mu]
=
[x_\eta,x_\mu]
$,
which belongs to $[M,M]$.
We will show that $N:=\mr{span}_{\cl{O}_K}([x_{n-1,0},x_{l1}]:
\, 0\les l<n)$
is equal to $L_1$, from which $L_1\subseteq [M,M]$ follows.
Since,
$
[x_{n-1,0},x_{l1}]
=
(\tau_{n-1}-\theta^{}(\tau_{n-1}))\tau_{l}\Pi
$,
we have $N=U\Pi$, where 
$U=
\mr{span}_{\cl{O}_K}([\tau_{n-1}-\theta^{}(\tau_{n-1})]\tau_{l}: 
0\les l<n)$.
Now the desired result follows from Remark \ref{rdownup} and 
the fact that $\overline{\tau}_{n-1}\not\in\kappa_K$.
\end{proof}

\begin{lemma}
\label{lcal2c}
Assume that $n\ges 3$,
$(p,n)\neq (3,3)$,
$\lambda_1=0$, and $\lambda_0= n-1$.
Then 
$L_1\subseteq [M,M]$.
\end{lemma}

\begin{proof} 
Let $i_0=1$ if $p\nmid n$, and $i_0=2$ if $p\,|\,n$.
Observe that $i_0<n-1$.
Denote $\eta=(i_0,0)$ and $\mu=(l,1)$, for 
$0\les l<n$.
Observe that $\eta< \lambda<\mu$, 
so
that 
$
[y_\eta,y_\mu]
=
[x_\eta,x_\mu]
+
e_\eta
[x_\lambda,x_\mu]
$,
which belongs to $[M,M]$.
By denoting $\xi_j=\tau_j-\theta(\tau_j)$ for $j\in \{i_0,n-1\}$,
we have 
$
[y_\eta,y_\mu]
=
\tau_l(\xi_{i_0}+e_{i_00}\xi_{n-1})\Pi
$.
We will show that $N:=\mr{span}_{\cl{O}_K}([y_{i_00},y_{l1}]:
\, 0\les l<n)$
is equal to $L_1$, from which $L_1\subseteq [M,M]$ follows.
We have $N=U\Pi$, where 
$U=
\mr{span}_{\cl{O}_K}(\tau_l(\xi_{i_0}+e_{i_00}\xi_{n-1}): 0\les l<n)$.
Observe that $\xi_{i_0}+e_{i_00}\xi_{n-1}\not\equiv_\pi 0$; see Remark \ref{rxisigma}.
Now the desired result follows from Remark \ref{rdownup}.
\end{proof}

\begin{lemma}
\label{lcal2d}
Assume that $\lambda_1=0$.
Then 
$\pi L_0\subseteq [M,M]$.
\end{lemma}

\begin{proof} 
For $n=1$ the result is trivial.
Assume $n\ges 2$.
Denote $\eta=(0,1)$ and $\mu=(l,n-1)$, for $0\les l<n$.
Observe that $\lambda<\eta$ and $\lambda <\mu$, 
so
that 
$
[y_\eta,y_\mu]
=
[x_\eta,x_\mu]
$,
which belongs to $[M,M]$.
We will show that $N:=\mr{span}_{\cl{O}_K}([x_{01},x_{l,n-1}]:
\, 0\les l<n)$
is equal to $\pi L_0$, from which $\pi L_0\subseteq [M,M]$ follows.
Since
$
[x_{01},x_{l,n-1}]
=
\pi  (\tau_{0}\theta(\tau_{l})-\tau_{l}\theta^{-1}(\tau_{0}))
$,
we have $N=\pi U$, where 
$U=
\mr{span}_{\cl{O}_K}(\tau_{0}\theta(\tau_{l})-\tau_{l}\theta^{-1}(\tau_{0}): 
0\les l<n)$.
Now the desired result follows from Remark \ref{rdownup} and Lemma \ref{lcyc0}.
\end{proof}

\begin{lemma}
\label{lcal3b}
Assume that $n\ges 3$,
$\lambda_1=1$, and
$\lambda_0<n-1$.
Then 
$\pi L_0\subseteq [M,M]$.
\end{lemma}

\begin{proof}
By considering the indices $\eta=(n-1,1)$ and $\mu=(l,n-1)$, for $0\les l<n$,
the proof is similar to the one of Lemma \ref{lcal2d}.
\end{proof}

\begin{lemma}
\label{lcal3d}
Assume that $n\ges 3$,
$\lambda_1=1$, and
$\lambda_0=n-1$.
Then 
$\pi L_0\subseteq [M,M]$.
\end{lemma}

\begin{proof} 
Denote $\eta=(0,1)$ and $\mu=(l,n-1)$, for 
$0\les l<n$,
and observe that 
$\eta<\lambda <\mu$.
Denote $\sigma_0=\tau_0+e_{\eta}\tau_{n-1}$,
and observe that $\sigma_0\not\equiv_\pi  0$.
After noticing that
$
[y_\eta,y_\mu]
=
[x_\eta,x_\mu]
+
e_\eta
[x_\lambda,x_\mu]
=
\pi (\sigma_0\theta(\tau_l)-\tau_l\theta^{-1}(\sigma_0))
$,
the proof proceeds as in the case of Lemma \ref{lcal2d}. 
\end{proof}

\begin{lemma}
\label{lcal0}
Assume that $\lambda_1\ges 1$.
Let $k$ and $m$ be integers such that $k\ges 0$ and $1\les m < n$,
and assume that $(\lambda_1,m)\neq (1,1)$.
Assume that $\pi^{k+1}L\subseteq [M,M]$
and $\pi^{k}L_j\subseteq [M,M]$ for $m<j<n$.
Then $\pi^k L_m\subseteq [M,M]$.
\end{lemma}

\begin{proof2}
Observe that we have $n\ges 3$.
We divide the proof into three cases.
\begin{enumerate}
\item
\label{lcal01}
Case $\lambda_1<m$. 
Denote $\eta=(i,0)$ and $\mu=(l,m)$,
for $1\les i <n$ and $0\les l<n$.
Observe that $\eta<\lambda<\mu$, so
that 
$
[y_\eta,y_\mu]
=
[x_\eta,x_\mu]
+
e_\eta
[x_\lambda,x_\mu]
$.
We also have 
$
[x_\lambda,x_\mu]\in L_{\lambda_1+m}$
when
$\lambda_1+m<n$,
and 
$
[x_\lambda,x_\mu]\in \pi L$
when
$\lambda_1+m\ges n$.
From the assumptions, 
since $\lambda_1+m > m$,
one can deduce that
$
\pi^k
[x_\eta,x_\mu]
\in [M,M]
$.
Observe that $[x_\eta,x_\mu]=[x_{i0},x_{lm}]=(\tau_i-\theta^m(\tau_i))\tau_l\Pi^m$.
We will show that $N_1:=\mr{span}_{\cl{O}_K}([x_{i0},x_{lm}]:\,1\les i<n\mbox{ and } 0\les l<n)$
is equal to $L_m$, from which $\pi^kL_m\subseteq [M,M]$ follows.
We have $N_1=U_1\Pi^m$, where 
$U_1=
\mr{span}_{\cl{O}_K}([\tau_i-\theta^m(\tau_i)]\tau_l:  \,1\les i<n\mbox{ and } 0\les l<n)$.
Now the desired result follows from Remark \ref{rdownup} and Lemma \ref{ltaneqa}.

\item
\label{lcal02}
Case $\lambda_1> m$. 
Denote $\eta=(i,0)$ and $\mu=(l,m)$,
for $1\les i <n$ and $0\les l<n$.
Observe that $\eta, \mu <\lambda$, 
so
that 
$
[y_\eta,y_\mu]
=
[x_\eta,x_\mu]
+
e_\eta
[x_\lambda,x_\mu]
+
e_\mu[x_\eta,x_\lambda]
$.
We have 
$
[x_\lambda,x_\mu]\in L_{\lambda_1+m}$
when
$\lambda_1+m<n$,
and 
$
[x_\lambda,x_\mu]\in \pi L$
when 
$\lambda_1+m\ges n$;
also,
$
[x_\eta, x_\lambda]\in L_{\lambda_1}$.
From the assumptions, 
since $\lambda_1+m>\lambda_1 > m$,
one can deduce that
$
\pi^k
[x_\eta,x_\mu]
\in [M,M]
$.
From here, the argument for this case is completed in exactly the same 
way as for case (\ref{lcal01}).

\item 
\label{lcal03}
Case $\lambda_1= m$. Observe that $m\ges 2$. 

Denote $\eta=(i,1)$ and $\mu=(l,m-1)$,
for $0\les i <n$ and $0\les l<n$.
Observe that $\eta, \mu <\lambda$, 
so
that 
$
[y_\eta,y_\mu]
=
[x_\eta,x_\mu]
+
e_\eta
[x_\lambda,x_\mu]
+
e_\mu[x_\eta,x_\lambda]
$.
We have 
$
[x_\lambda,x_\mu]\in L_{2m-1}$
when
$2m-1<n$,
and 
$
[x_\lambda,x_\mu]\in \pi L$
when 
$2m-1\ges n$;
also,
$
[x_\eta, x_\lambda]\in L_{m+1}$
when $m+1<n$, and 
$
[x_\eta, x_\lambda]\in \pi L$
when $m+1\ges n$. 
From the assumptions, 
since $2m-1\ges  m+1 > m$,
one can deduce that
$
\pi^k
[x_\eta,x_\mu]
\in [M,M]
$.
After noticing that
$[x_\eta,x_\mu]=[x_{i1},x_{l,m-1}]=
(\tau_{i}\theta(\tau_{l})-\tau_{l}\theta^{m-1}(\tau_{i}))\Pi^m$,
the proof proceeds as in the case of Lemma \ref{lcal2}.\ep 
\end{enumerate}
\end{proof2}

\begin{lemma}
\label{lcal1}
Assume that $n\ges 3$, $(p,n)\neq (3,3)$,
and $\lambda_1 = 1$.
Let $k$ be an integer such that $k\ges 0$,
and assume that 
$\pi^{k}L_2\subseteq [M,M]$.
Then $\pi^k L_1\subseteq [M,M]$.
\end{lemma}

\begin{proof2}
Denote $\eta=(i,0)$ and $\mu=(l,1)$,
for $1\les i <n$, $0\les l<n$,
and $l\neq \lambda_0$.
Observe that $\eta<\lambda$;
also, 
$\lambda<\mu$ when $l>\lambda_0$,
and
$\mu<\lambda$ when $l<\lambda_0$.
Define 
$
z_{il}= 
[x_{i0},x_{l1}]
$
when $l>\lambda_0$,
and 
$
z_{il}= 
[x_{i0},x_{l1}]
+
e_l
[x_{i0},x_\lambda]
$
when $l<\lambda_0$,
where $e_l=e_{l1}$. 
Hence,
$
[y_\eta,y_\mu]
=
z_{il}
+
e_\eta
[x_\lambda,x_\mu]
$.
We have 
$
[x_\lambda,x_\mu]\in L_{2}$.
From the assumptions
one can deduce that
$
\pi^k
z_{il}
\in [M,M]
$.
For $1\les i <n$ 
we define 
$\xi_i=\tau_i-\theta(\tau_i)$, while for 
$0\les l < n$ and $l\neq \lambda_0$
we define 
$\sigma_l = \tau_l$ when $l>\lambda_0$,
and
$\sigma_l = \tau_l+e_l\tau_{\lambda_0}$ when $l<\lambda_0$.
Observe that
$z_{il}=
\xi_i\sigma_l\Pi$.  
We will show that 
$N:=\mr{span}_{\cl{O}_K}(z_{il}:\,1\les i<n, \,0\les l<n, \mbox{ and } l\neq\lambda_0)$
is equal to $L_1$, from which $\pi^kL_1\subseteq [M,M]$ follows.
We have $N=U\Pi$, where 
$U=
\mr{span}_{\cl{O}_K}(\xi_i\sigma_l:  \,1\les i<n,\, 0\les l<n, \mbox{ and } l\neq\lambda_0)$.
Now we want to apply Remark \ref{rdownup}.
In order to do so, we observe that $(\overline{\sigma}_l)_{l\neq \lambda_0}$ spans
a subspace of dimension $n-1$ over $\kappa_K$. 
Hence, the desired result follows from Remark \ref{rxisigma} and Remark \ref{rcyc2}.

\ep
\end{proof2}

\begin{lemma}
\label{lcal3a}
Assume that 
$2\les \lambda_1\les n-2$,
and let $k$ be an integer such that $k\ges 1$.
Assume that $\pi^{k}L_j\subseteq [M,M]$ for $0<j<n$.
Then 
$\pi^kL_0\subseteq [M,M]$.
\end{lemma}

\begin{proof}
Observe that $n\ges 4$. 
Denote $\eta=(0,1)$ and $\mu=(l,n-1)$,
for $0\les l<n$.
Observe that $\eta<\lambda<\mu$, so
that 
$
[y_\eta,y_\mu]
=
[x_\eta,x_\mu]
+
e_\eta
[x_\lambda,x_\mu]
$.
Since $n<\lambda_1+n-1<2n$,
we have 
$
[x_\lambda,x_\mu]
\in 
\pi L_{\lambda_1-1}
$.
From the assumptions,
since $0<\lambda_1-1<n$,
one can deduce that
$
\pi^{k-1}
[x_\eta,x_\mu]
\in [M,M]
$.
After noticing that
$[x_\eta,x_\mu]=[x_{01},x_{l,n-1}]=
\pi (\tau_{0}\theta(\tau_{l})-\tau_{l}\theta^{-1}(\tau_{0}))$,
the proof proceeds as in the case of Lemma \ref{lcal2d}.
\end{proof}

\begin{lemma}
\label{lcal3ce}
Assume that $n\ges 3$ and
$\lambda_1 = n-1$.
Let $k$ be an integer such that $k\ges 1$,
and assume that $\pi^{k}L_j\subseteq [M,M]$ for $0<j<n$.
Then 
$\pi^kL_0\subseteq [M,M]$.
\end{lemma}

\begin{proof2}
We divide the proof into two parts.
\begin{enumerate}
\item 
\label{lcal3ce.1}

Case $\lambda_0<n-1$.
Denote $\eta=(i,1)$ and $\mu=(n-1,n-1)$,
for $0\les i <n$.
Observe that $\eta<\lambda<\mu$, so
that 
$
[y_\eta,y_\mu]
=
[x_\eta,x_\mu]
+
e_\eta
[x_\lambda,x_\mu]
$.
Since $\lambda_1+\mu_1=2n-2$ and
in particular $n<\lambda_1+\mu_1<2n$,
we have 
$
[x_\lambda,x_\mu]
\in 
\pi L_{n-2}
$.
From the assumptions,
since $0<n-2<n$,
one can deduce that
$
\pi^{k-1}
[x_\eta,x_\mu]
\in [M,M]
$.
Observe that $[x_\eta,x_\mu]=[x_{i1},x_{n-1,n-1}]=
\pi (\tau_{i}\theta(\tau_{n-1})-\tau_{n-1}\theta^{-1}(\tau_{i}))$.
We will show that $N:=\mr{span}_{\cl{O}_K}([x_{i1},x_{n-1,n-1}]:\,0\les i<n)$
is equal to $\pi L_0$, from which $\pi^kL_0\subseteq [M,M]$ follows.
We have $N=\pi U$, where 
$U=
\mr{span}_{\cl{O}_K}(\tau_{i}\theta(\tau_{n-1})-\tau_{n-1}\theta^{-1}(\tau_{i}):  \,0\les i <n)$.
Now the desired result follows from Remark \ref{rdownup} and Remark \ref{rcyc0}.

\item 
Case $\lambda_0=n-1$.
Denote $\eta=(i,1)$ and $\mu=(0,n-1)$,
for $0\les i <n$.
Observe that $\eta<\lambda$ and $\mu<\lambda$, so
that 
$
[y_\eta,y_\mu]
=
[x_\eta,x_\mu]
+
e_\eta
[x_\lambda,x_\mu]
+
e_\mu[x_\eta,x_\lambda]
$.
Since $\lambda_1+\mu_1=2n-2$ and
in particular $n<\lambda_1+\mu_1<2n$,
we have 
$
[x_\lambda,x_\mu]
\in 
\pi L_{n-2}
$.
Define $z_i= 
[x_\eta,x_\mu]
+
e_\mu[x_\eta,x_\lambda]
$
for $0\les i <n$,
and 
define 
$\sigma_0=\tau_0+e_{\mu}\tau_{n-1}$.
Observe that $\sigma_0\not\equiv_\pi  0$.
From the assumptions,
since $0<n-2<n$,
one can deduce that
$
\pi^{k-1}z_i\in [M,M]
$.
After noticing that
$
z_i
=
\pi (\tau_i\theta(\sigma_0)-\sigma_0\theta^{-1}(\tau_i))
$,
the proof proceeds as in the case of item (\ref{lcal3ce.1}).\ep 
\end{enumerate}
\end{proof2}

\begin{corollary}
\label{cl01}
Assume that $n\ges 3$,
$(p,n)\neq (3,3)$,
 and
$\lambda_1 \les 1$.
Then 
$[M,M]=[L,L]$.
\end{corollary}

\begin{proof2}
We divide the proof into two parts.
\begin{enumerate}
\item 
Case $\lambda_1=0$.
The result follows from
Remark \ref{rp2llmm} and Lemmas \ref{lcal2}, \ref{lcal2b}, \ref{lcal2c}, 
and \ref{lcal2d}.

\item 
Case $\lambda_1=1$.
We will prove that if 
$k\ges 0$ is an integer such that 
$\pi^{k+1}[L,L]\subseteq [M,M]$
then $\pi^{k}[L,L]\subseteq [M,M]$.
Since  $\pi^{2}[L,L]\subseteq [M,M]$,
see Remark \ref{rp2llmm},
the corollary follows by descending induction on $k$.
Let $k\ges 0$ be an integer such that 
$\pi^{k+1}[L,L]\subseteq [M,M]$.
We have $\pi^{k+1}L_j\subseteq [M,M]$ for $0<j<n$.
By applying Lemma \ref{lcal3b} in case $\lambda_0<n-1$,
and by applying Lemma \ref{lcal3d} in case $\lambda_0=n-1$,
we see that $\pi^{k+1}L_0\subseteq [M,M]$.
Hence, $\pi^{k+1}L\subseteq [M,M]$.
By Lemma \ref{lcal0} and descending induction on $m$,
$\pi^kL_m\subseteq [M,M]$ for $2\les m <n$.
By Lemma \ref{lcal1}, $\pi^kL_1\subseteq [M,M]$.
It follows that $\pi^k[L,L]\subseteq [M,M]$.\ep 
\end{enumerate}
\end{proof2}

\begin{corollary}
\label{cl2n}
Assume that 
$\lambda_1\ges 2$. 
Then 
$[M,M]=[L,L]$.
\end{corollary}

\begin{proof}
Observe that $n\ges 3$.
We will prove that if 
$k\ges 0$ is an integer such that 
$\pi^{k+1}[L,L]\subseteq [M,M]$
then $\pi^{k}[L,L]\subseteq [M,M]$.
Since  $\pi^{2}[L,L]\subseteq [M,M]$,
see Remark \ref{rp2llmm},
the corollary follows by descending induction on $k$.
Let $k\ges 0$ be an integer such that 
$\pi^{k+1}[L,L]\subseteq [M,M]$.
We have $\pi^{k+1}L_j\subseteq [M,M]$ for $0<j<n$.
By applying Lemma \ref{lcal3a} in case $\lambda_1<n-1$,
and by applying Lemma \ref{lcal3ce} in case $\lambda_1=n-1$,
we see that $\pi^{k+1}L_0\subseteq [M,M]$.
Hence, $\pi^{k+1}L\subseteq [M,M]$.
By Lemma \ref{lcal0} and descending induction on $m$,
$\pi^kL_m\subseteq [M,M]$ for $1\les m <n$.
It follows that $\pi^k[L,L]\subseteq [M,M]$.
\end{proof}

%%%%%%%%%%%%%%%%%%
\subsection{Proof of the theorems}
\label{proofCDEss}

\textbf{Proof of Theorem \ref{tMMLLZ_p.2}}.
The theorem follows from 
Remark \ref{rf2noO_K}, Lemma \ref{lnotO_KNNLL}, 
and Theorem \ref{tMMLLO_K}. \ep

\bigskip

\noindent
\textbf{Proof of Theorem \ref{taim2.2}}.
Denote $L=sl_1(\Delta)$.
In case $n=1$ we have $L=\{0\}$, and the result is trivial. 
Assume that $n\ges 2$, in which case $L$ is just infinite
as a $\bb{Z}_p$-Lie lattice
(Remark \ref{rsimple}) and $\mr{dim}_{\bb{Z}_p}(L)\ges 3$.
In case $f\ges 2$, the result follows from 
Theorem \ref{tMMLLZ_p.2} and Lemma \ref{lMMLLnss2}. 
Assume that $(p,n)\not\in\{(2,2),(3,3)\}$. 
In case $n=2$ we have $p\neq 2$, and the result follows from Theorem \ref{tn2}. 
In case $n\ges 3$ we have $(p,n)\neq (3,3)$,
and the result follows from Theorem \ref{tMMLLZ_p.2} and  Lemma \ref{lMMLLnss2}.
\ep

\bigskip

\noindent 
\textbf{Proof of Theorem \ref{tmainalt}}.
For $n=1$, the group $G:=SL_1^{mne}(D)$ is trivial,
and the statements are true.
Assume that $n\ges 2$, and note that $p\ges 3$.
By \cite[Proposition 4.3 and Proposition 4.4]{Ers2CohN1_arxiv},
for $l\ges ne$,
$SL_1^l(D)$ is a uniform pro-$p$ group
and its associated powerful $\bb{Z}_p$-Lie lattice 
is $sl_1^l(\Delta) := sl_1(\Delta)\cap \mfr{p}^l$.
Since $sl_1^{mne}(\Delta)= p^msl_1(\Delta)$, 
the powerful $\bb{Z}_p$-Lie lattice associated with $G$ is $L_G= p^msl_1(\Delta)$.
Now, item (\ref{tmainalt1}) follows from \cite[Proposition A]{NS2019} and Theorem \ref{taim2.2}. 
For item (\ref{tmainalt2}), observe that $n\ges 3$ and $p\ges 11$;
moreover, since $G$ is a finitely generated pro-$p$ group,
$H$ has index $p$ in $G$.
By \cite[Theorem 3.1]{NS2019}, $H$ is saturable and the index of $L_H$
in $L_G$ is $p$, where $L_H$ is the $\bb{Z}_p$-Lie lattice associated
with $H$.
Also, by \cite[Theorem B]{GSpsat}, $[L_G,L_G]=[G,G]$ and $[L_H, L_H]= [H,H]$.
Observe that $p^{-m}L_H$ is a submodule of $p^{-m}L_G= sl_1(\Delta)$
of index $p$.
By Theorem \ref{tMMLLZ_p.2}, $[p^{-m}L_H, p^{-m}L_H]=[sl_1(\Delta),sl_1(\Delta)]$,
hence, $[L_H, L_H]=[L_G, L_G]$. The desired $[H,H]=[G,G]$ follows.\ep

%%%%%%%%%%%%%%%%%%%
%%%%%%%%%%%%%%%%%%%

\section{Proof of Theorem  \ref{tshssL}, Theorem \ref{conjE}, and Theorem \ref{thBext}}
\label{proofGHI}

The `only if' part of Theorem \ref{tshssL} follows from \cite[Theorem 2.34]{NSGGD22} and
Theorem \ref{tunsnshss} below. 
The `if' part follows from \cite[Remark 2.35, Proposition 2.40]{NSGGD22}. 

\begin{theorem}
\label{tunsnshss}
Assume that $p\ges 5$, and let $L$ be an unsolvable $\bb{Z}_p$-Lie lattice.
Then $L$ is not strongly hereditarily self-similar of index $p$.
\end{theorem}

\begin{proof}
Observe that $d:=\mr{dim}(L) \ges 3$.
We will prove that there exists
a nonzero $\bb{Z}_p$-subalgebra $L_0$ of $L$ that is not self-similar of index $p$.
Let $\cl{L} = L\otimes_{\bb{Z}_p}\bb{Q}_p$, which is an unsolvable $\bb{Q}_p$-Lie
algebra of dimension $d$. 
By the Levi decomposition theorem 
\cite[page 91]{JacLieA} and by the unsolvability of $\cl{L}$,
there exists a nonzero semisimple $\bb{Q}_p$-subalgebra $\cl{L}_1$ of $\cl{L}$.
By the structure theorem of semisimple Lie algebras 
\cite[page 71]{JacLieA}, there exists a simple $\bb{Q}_p$-subalgebra $\cl{M}$ of $\cl{L}_1$. 
We divide the proof into two cases according whether $\cl{M}$ contains a nonzero
nilpotent element or not. 

First, we assume that $\cl{M}$ contains a nonzero nilpotent element $e$. 
By the Jacobson-Morozov theorem (see \cite[Chapter VIII, Section 11, Proposition 2]{BouLieGrAlg3}), 
the element $e$ may be completed to an $sl_2$-triple, that is,
there exist $f,h\in\cl{M}$ such that
the set $\{e,f,h\}$ is linearly independent over $\bb{Q}_p$,
$[h,e]=2e$, $[h,f]=-2f$, and $[e,f]=h$. Let ${M}$ be the $\bb{Z}_p$-subalgebra
of $\cl{L}$ generated by $e$, $f$, and $h$, and observe that $M$ is a $\bb{Z}_p$-Lie lattice
isomorphic to $sl_2(\bb{Z}_p)$; moreover, there exists $m\in\bb{N}$ such that $p^mM\subseteq L$.
By \cite[Proposition 2.41]{NSGGD22}, there exists a nonzero $\bb{Z}_p$-subalgebra $L_0$ of $p^mM$ which is not
self-similar of index $p$.

Second, we assume that $\cl{M}$ contains no nonzero nilpotent elements.
By \cite[Lemma 3.2.1]{Schoeneberg2014},
$\cl{M}$ is anisotropic;
hence,
by \cite[Theorem (2), page 6]{Schoeneberg2014},
$\cl{M}$ is isomorphic, as $\bb{Q}_p$-Lie algebra, 
to the derived algebra $[D,D]$
of a finite dimensional division $\bb{Q}_p$-algebra $D$. 
Let $K$ be the center of $D$. 
Hence, $K$ is a finite field extension of $\bb{Q}_p$,
and $D$ is a {central} division $K$-algebra.
Let $\cl{O}_K$ be the integral closure of $\bb{Z}_p$ in $K$,
and let $\Delta$ be the unique maximal $\cl{O}_K$-order in $D$.
Observe that, since $sl_1(D) = [D,D]$,
$sl_1(\Delta)$ is a $\bb{Z}_p$-Lie subalgebra of $[D,D]$.  
Let $M$ be the $\bb{Z}_p$-Lie lattice inside $\cl{M}$ that corresponds to $sl_1(\Delta)$
via an isomorphism $\cl{M}\rar sl_1(D)$. There exists $m\in\bb{N}$ such that
$p^mM\subseteq L$. By Theorem \ref{taim2.2}, we can take $L_0=p^mM$. 
\end{proof}

\bigskip

\noindent
\textbf{Proof of Theorem \ref{conjE}}.
We denote by $L^d(s)$ and $G^d(s)$ the Lie lattices and groups appearing in
item (2) of the statements of Theorem \ref{tshssL} and Theorem \ref{conjE}, respectively.
We proceed with the proof by first assuming that $G$ is isomorphic to one of the groups listed in the statement.
Then $G$ is solvable, and we can apply \cite[Theorem A]{NSGGD22} to deduce the desired result
(see also \cite[Proposition 3.7]{NSGGD22}).
Now assume that $G$ is
strongly hereditarily self-similar of index $p$. 
If $d\les 2$ then $G$ is solvable, and again we can apply \cite[Theorem A]{NSGGD22}.
We assume that $d\ges 3$ and observe that in this case $p\ges 5$.
By \cite[Remark 3.2]{NSGGD22}, we can apply 
\cite[Proposition 3.1]{NSGGD22}
and deduce that $G$ is saturable and the $\bb{Z}_p$-Lie lattice $L$ associated with $G$ is
strongly hereditarily self-similar of index $p$. 
By Theorem \ref{tshssL}, $L$ is isomorphic to either $\bb{Z}_p^d$ or to $L^d(s)$
for some $s\in\bb{N}$;
in fact, in the latter case we have $s\ges 1$, since $L$ is residually nilpotent
by \cite[Theorem B]{GSKpsdimJGT}. 
Then, since the group associated with $\bb{Z}_p^d$ (respectively, $L^d(s)$) is 
$\bb{Z}_p^d$ (respectively, $G^d(s)$), see \cite[Remark 3.6]{NSGGD22},
$G$ is isomorphic to either $\bb{Z}_p^d$ or $G^d(s)$, as desired.

\ep

\bigskip

\noindent
\textbf{Proof of Theorem \ref{thBext}}.
Denote $d=\mr{dim}(G)$, and consider the following property.
\begin{enumerate}
\setcounter{enumi}{-1}
\item 
\label{peqc.0}
$G\simeq \bb{Z}_p^d$,
or $d\ges2$ and there exists an integer $s\ges 1$ such that 
$G\simeq \bb{Z}_p \ltimes \bb{Z}_p^{d-1}$, where $\bb{Z}_p$
acts on $\bb{Z}_p^{d-1}$ by multiplication by $1+p^s$.
\end{enumerate}
The equivalences 
(\ref{peqc.0})$\Leftrightarrow$(\ref{peqc.1}),
(\ref{peqc.0})$\Leftrightarrow$(\ref{peqc.3}),
(\ref{peqc.0})$\Leftrightarrow$(\ref{peqc.5}), and
(\ref{peqc.0})$\Leftrightarrow$(\ref{peqc.6})
follow from 
Theorem \ref{conjE},
\cite[Corollary 2.4]{KSjalg11}, 
\cite[Corollary 1.13]{KSquart14}, and 
\cite[Theorem 1.2]{ST2020frattini},
respectively.
The implication 
(\ref{peqc.0})$\Rightarrow$(\ref{peqc.2})
is a consequence of \cite[Theorem A, Theorem B]{NSGGD22}
and the fact that the groups appearing in item (\ref{peqc.0}) are solvable.
The implication
(\ref{peqc.2})$\Rightarrow$(\ref{peqc.4})
follows from the Norm Residue Theorem (cf. \cite{HWnormresth}).
For $p=2$, the implication
(\ref{peqc.4})$\Rightarrow$(\ref{peqc.0}) 
follows from the assumption $p>d$, so that $d=1$ and $G\simeq\bb{Z}_2$; 
on the other hand, for $p\neq 2$, the implication follows from
\cite[Theorem B, Proposition 3.4]{Quad14}
and the fact that a $p$-adic analytic pro-$p$ group does not admit
a nonabelian free pro-$p$ subgroup.\ep

%%%%%%%%%%%%%%%%%%%%%%%%%%%%%%%% References %%%%%%%%%%%%%%%%%%%%%%%%%%%%%%%%%%%%%%%%%%%%%%%%%%%%%%%%%%

\begin{footnotesize}

\providecommand{\bysame}{\leavevmode\hbox to3em{\hrulefill}\thinspace}
\providecommand{\MR}{\relax\ifhmode\unskip\space\fi MR }
% \MRhref is called by the amsart/book/proc definition of \MR.
\providecommand{\MRhref}[2]{%
  \href{http://www.ams.org/mathscinet-getitem?mr=#1}{#2}
}
\providecommand{\href}[2]{#2}

\end{footnotesize}

\vspace{5mm}
\newpage

\noindent
Francesco Noseda\\
Mathematics Institute\\
Federal University of Rio de Janeiro\\
Av. Athos da Silveira Ramos, 149\\
21941-909, Rio de Janeiro, RJ\\
Brazil  \\
{\tt noseda@im.ufrj.br}\\

\noindent
Ilir Snopce \\
Mathematics Institute\\
Federal University of Rio de Janeiro\\
Av. Athos da Silveira Ramos, 149\\
21941-909, Rio de Janeiro, RJ\\
Brazil  \\
{\tt ilir@im.ufrj.br}\\

\noindent
\today
\end{document}